\documentclass[11pt,a4paper]{amsart}

\usepackage{amssymb,amscd,amsmath,young, stmaryrd,tikz-cd}
\usepackage{mathrsfs, mathdots}
\usepackage[mathcal]{eucal}
\usepackage{float,mathabx}
\usepackage{soul,caption}
\usepackage{cancel}
\usepackage[normalem]{ulem}
\usetikzlibrary{arrows}

\usepackage{longtable,array,multirow,makecell,graphicx}

\theoremstyle{plain}
\newtheorem{thm}{Theorem}[section]
\newtheorem{lem}[thm]{Lemma}

\newtheorem{cor}[thm]{Corollary}
\newtheorem*{claim*}{Claim}

\newtheorem{con}[thm]{Conjecture}

\theoremstyle{remark}

\newtheorem{rem}[thm]{Remark}
\newtheorem{qun}[thm]{Question}
\newtheorem{exm}[thm]{Example}

\newtheorem{dfn}[thm]{Definition}

\numberwithin{equation}{section}

\numberwithin{table}{section}

\newcommand{\N}{\mathbb{N}}
\newcommand{\Z}{\mathbb{Z}}
\newcommand{\Q}{\mathbb{Q}}
\newcommand{\F}{\mathbb{F}}
\newcommand{\C}{\mathbb{C}}

\newcommand{\tensor}{\otimes}
\newcommand{\spl}{\mathfrak{sl}}
\newcommand{\gl}{\mathfrak{gl}}

\newcommand{\mff}{\mathfrak{f}}

\renewcommand{\epsilon}{\varepsilon}

\renewcommand{\phi}{\varphi}
\renewcommand{\theta}{\vartheta}

\newcommand{\Qp}{\mathbb{Q}_{p}}

\newcommand{\ideal}{\triangleleft}

\newcommand{\Zp}{\mathbb{Z}_{p}}
\newcommand{\bsm}{\boldsymbol{m}}

\newcommand{\LFp}{L(\mathbb{F}_{p})}

\DeclareMathOperator{\End}{End}

\DeclareMathOperator{\gr}{gr}

\DeclareMathOperator{\Gr}{Gr}
\DeclareMathOperator{\ad}{ad}

\DeclareMathOperator{\rk}{rk}
\DeclareMathOperator{\SL}{SL}

\def \idealgr {\triangleleft_\textup{gr}}

\def \Fp {\ensuremath{\mathbb{F}_p}}

\def \Fq {\ensuremath{\mathbb{F}_q}}
\def \Zp  {\mathbb{Z}_p}

\begin{document}
	
	\title[Zeta functions of $\Fp$-Lie algebras and finite $p$-groups]{Zeta functions of $\Fp$-Lie algebras and finite $p$-groups}
	\author{Seungjai Lee} \address{The Research Institute of Basic Sciences, Seoul National University\\
		Seoul, 08826, South Korea}
		\email{seungjai.lee@snu.ac.kr}

	

\begin{abstract}
	We study zeta functions enumerating subalgebras or ideals of Lie algebras over finite field of prime order $\mathbb{F}_p$.  We first develop a general blueprint  method for computing zeta functions of $\mathbb{F}_p$-Lie algebras, and demonstrate its practical applications in detail to obtain explicit formulas for various interesting new examples that are not covered in any known literature yet. For nilpotent cases this also provides zeta functions counting subgroups and normal subgroups of finite $p$-groups of exponent $p$ for almost all primes via the Lazard correspondence. We investigate their connections to the study of finite $p$-groups, and discuss what can be deduced from these finite Dirichlet polynomials.  
\end{abstract}

	\maketitle

	\thispagestyle{empty}

\section{Introduction}
\subsection{Background and motivations}
Let $L$ be a ring, additively isomorphic to $\Z^{n}$ for some $n$, and for $*\in\{\leq,\triangleleft\}$, let 
\[a_{m}^{*}(L):=|\{H*L\mid|L:H|=m\}|.\]
In their seminal paper \cite{GSS/88}, Grunewald, Segal, and Smith defined the \textit{subalgebra ($*=\leq$)} and \textit{ideal ($*=\triangleleft$) zeta functions of $L$} to be the Dirichlet generating series
\[\zeta_{L}^{*}(s):=\sum_{m=1}^{\infty}a_{m}^{*}(L)m^{-s},\]
where $s$ is a complex variable (here, and in what follows, $*$ stands for $\leq$ or $\triangleleft$).

This admits a natural Euler decomposition \[\zeta_{L}^{*}(s)=\prod_{p\textrm{ prime}}\zeta_{L,p}^{*}(s),\] where $\zeta_{L,p}^{*}(s)=\sum_{i=0}^{\infty}a_{p^{i}}^{*}(L)p^{-is}$. 
Note that for each prime $p$, the local factor $\zeta_{L,p}^{*}(s)$ is also the zeta function of the $\Zp$-algebra $L(\Zp):=L\tensor\Zp$, where $\Zp$ denotes the ring of $p$-adic integers. We call these local factors the \textit{local subalgebra/ideal zeta functions of $L$}. It is known that each of these local factors $\zeta_{L(\Zp)}^{*}(s)$ is a rational function over $\Q$ in $p^{-s}$ (\cite[Theorem 3.5]{GSS/88}). Over the last few decades, zeta functions of Lie rings have become important tools in asymptotic group and ring theory. Good surveys of the topics are given in \cite{duS-ennui/02,duSWoodward/08,Voll/11,Voll/15}.

One of the fundamental questions in this field is  how these (rational) local factors $\zeta_{L(\Zp)}^{*}(s)$ behave as we vary $p$: called the \textit{uniformity} problem. 

\begin{dfn}\label{def:uniform.zp}
	The (global) zeta function $\zeta_{L}^{*}(s)$ is \textit{finitely uniform} if there exist rational functions $W_{1}(X,Y),\ldots,W_{k}(X,Y)\in\Q(X,Y)$ for $k\in\mathbb{N}$ such that, for every prime $p$, 
	\[\zeta_{L(\Zp)}^{*}(s)=W_{i}^{*}(p,p^{-s})\]
	for some $i\in\{1,\ldots,k\}$. It is  \textit{uniform} if $k=1$, and \textit{non-uniform} if it is not finitely uniform.  
\end{dfn}

In \cite{duSG/00}, du Sautoy and Grunewald proved that the behavior of local factors are related to the classical problem of counting points on varieties mod $p$. 

\begin{thm}[\cite{duSG/00}, Theorem 1.3]\label{thm:cone.int.p}
	Let $L$ be a nilpotent Lie ring. There exists an algebraic variety $Y^{*}$ defined over $\mathbb{Q},$ with irreducible subvarieties $E_{i}^{*},$ $i\in T^{*}:=\{1,\,\ldots,\,t^{*}\},$
	all of which are smooth and intersect normally, and rational functions
	$P_{I}^{*}(X,\,Y)\in\mathbb{Q}(X,\,Y),$ $I\subseteq T^{*}$ such
	that, for almost all primes $p,$ 
\begin{equation}\label{eq:duSG}
	\zeta_{L,p}^{*}(s)=\sum_{I\subseteq T^{*}}c_{p,\,I}^{*}P_{I}^{*}(p,\,p^{-s}),
\end{equation}
	where $c_{p,\,I}^{*}=\left|\left\{ a\in\overline{Y^{*}}(\mathbb{F}_{p}):a\in\overline{E_{i}^{*}}\textrm{\;\ if and only if }i\in I\right\} \right|$.
\end{thm}

Counting points on varieties mod $p$ is a very wild problem in general. For instance, du Sautoy  showed in \cite{duS-ecI/01} that there exists a class-2 nilpotent Lie ring $L_E$ of rank 9 with non-uniform zeta functions. Its local zeta functions $\zeta_{L_E,p}^{*}(s)$ encodes the arithmetic of an elliptic curve. It is still a mystery to understand what kind of varieties can appear in the description of zeta functions of Lie rings given in \eqref{eq:duSG}.

One of the main obstacles to understand the nature of $\zeta_{L(\Zp)}^{*}(s)$ is that they are in general too difficult to compute. Still, together with Stanley's result on rational functions \cite[Theorem 4.1.1]{Stanley/12}, the rationality of $\zeta_{L,p}^{*}(s)$  implies that the sequence $a_{p^{i}}^{*}(L)$ satisfies a strong regularity property, i.e., the number of finite-index subalgebras or ideals of $L(\Zp)$ are already determined by the number of subalgebras or ideals in some finite quotients of $L(\Zp)$.  

\subsection{Main results and organization} 

In this light, the aim of this article is to introduce and initiate the study of zeta functions enumerating subalgebras or ideals of various $\Fp$-Lie algebras.

Let $\Fq$ denote a finite field on $q=p^r$ elements, where $p$ is a prime, and let $A$ denote the $\Fq$-algebra, additively isomorphic to $\Fq^{n}$. One can analogously define the \textit{(subalgebra or ideal) zeta functions of $\Fq$-algebras} as a finite Dirichlet polynomial
\begin{equation}\label{eq:fp.zeta}
	\zeta_{A}^{*}(s):=\sum_{B*A}|A:B|^{-s}=\sum_{i=0}^{n}a_{q^{i}}^{*}(A)q^{-is},
\end{equation}
where $a_{q^{i}}(A)$ denote the number of subalgebras or ideals of $A$ of codimension $i$.
	
\begin{exm}
	For $n\in\N$, we have
		\begin{align*}
			\zeta_{\Fq^{n}}(s):=\zeta_{\Fq^{n}}^{\leq}(s)=\zeta_{\Fq^{n}}^{\ideal}(s)&=\sum_{i=0}^{n}a_{q^{i}}^{*}(\Fq^{n})q^{-is}=\sum_{i=0}^{n}\binom{n}{i}_{q}\,q^{-is},
		\end{align*}
		where 
		\[\binom{n}{i}_{q}=\begin{cases}\frac{(1-q^n)(1-q^{n-1})\cdots(1-q^{n-i+1})}{(1-q)(1-q^2)\cdots(1-q^i)}&i\leq n,\\0&i>n,\end{cases}\] is the Gaussian binomial coefficient that counts the number of subspaces of (co)dimension $i$ in a vector space of dimension $n$ over $\Fq$. 
\end{exm}

In this article, we concentrate on the situation where $q=p$, $L\cong\Z^{n}$ is a Lie ring, and $A=L(\Fp)$ is an $\Fp$-Lie algebra, additively isomorphic to $\Fp^{n}$. In this case, one can view $\zeta_{L(\Fp)}^{*}(s)$ as the ``simplest finite approximations" of $\zeta_{L(\Zp)}^{*}(s)$ (which is why we de define them as Dirichlet polynomials), in the sense that $\Fp=\Z/p\Z$ can be seen as the ``first layer" of $\Zp={\varprojlim}\Z/p^{i}\Z$, and all the subobjects enumerated by $\zeta_{L(\Fp)}^{*}(s)$ are also enumerated by $\zeta_{L(\Zp)}^{*}(s)$. For instance, from the computations in this article one can observe  that if both $\zeta_{L(\Zp)}^{*}(s)$ and $\zeta_{L(\Fp)}^{*}(s)$ are known, then 
\[a_{p^{i}}^{*}(L(\Fp))\leq a_{p^{i}}^{*}(L(\Zp))\]
for $0\leq i \leq n$ as expected, with $a_{1}^{*}(L(\Fp))=a_{1}^{*}(L(\Zp))=1$ and $a_{p}^{*}(L(\Fp))=a_{p}^{*}(L(\Zp))$.

Also, let us analogously define the $\Fp$-uniformity as follows:\begin{dfn}\label{def:uniform.fp}
	The (global) zeta function $\zeta_{L}^{*}(s)$ is $\Fp$-\textit{finitely uniform} if there exist polynomials $W_{1}^{*}(X,Y),\ldots,W_{k}^{*}(X,Y)\in\Q[X,Y]$ for $k\in\mathbb{N}$ such that, for every prime $p$, 
	\[\zeta_{L(\Fp)}^{*}(s)=W_{i}^{*}(p,p^{-s})\]
	for some $i\in\{{1,\ldots,k}\}$, $\Fp$-\textit{uniform} if $k=1$, and $\Fp$-\textit{non-uniform} if it is not $\Fp$-finitely uniform.
\end{dfn}

In this way the `classical' uniformity property, as defined in Definition \ref{def:uniform.zp}, can be considered as `$\Zp$-uniformity'. Comparing the behavior of $\zeta_{L(\Fp)}^{*}(s)$ and $\zeta_{L(\Zp)}^{*}(s)$ over $p$ would tell us whether ``all the actions happen at the top or not". 
As we show later, in many cases (e.g. Theorem \ref{thm:Elliptic}) studying $\zeta_{L(\Fp)}^{*}(s)$ was enough to predict the behavior of $\zeta_{L(\Zp)}^{*}(s)$. In particular, so far we could not find a single example of a Lie ring $L$ whose zeta function $\zeta_{L}^{*}(s)$ has different uniformity over $\Zp$ and $\Fp$, suggesting the usefulness of $\zeta_{L(\Fp)}^{*}(s)$ as an approximation of $\zeta_{L(\Zp)}^{*}(s)$.

This article is arranged as follows. In Section \ref{sec:2}, by translating the enumeration problems of subalgebras and ideals of fixed index into the calculations of certain systems of equations, we give a constructive and explicit method to compute $\zeta_{L(\Fp)}^{*}(s)$ in the most general setting, and compare this with \cite[Theorem 1.3]{duSG/00} on  $\zeta_{L(\Zp)}^{*}(s)$. We then show how one can directly use this method to compute the zeta functions of the Heisenberg, $\spl_2$, and $\gl_2$. In Section \ref{sec:3}, we first show in Section \ref{subsec:nil.fp} how one can modify and use this method developed in Section \ref{sec:2} in fortunate cases, even for some infinite families of $\Fp$-Lie algebras. Then, with Lemma \ref{lem:idealc2}, we discuss in Section \ref{subsec:c2} how things get better when we restrict our $L$ to be class-2 nilpotent Lie rings. In Section \ref{sec:4} and Section \ref{sec:5} we further record explicit computations of $\zeta_{L(\Fp)}^{*}(s)$ for various interesting $\Fp$-Lie algebras. In particular:

\begin{itemize}
	\item We compute subalgebra zeta functions for a semisimple $\Fp$-Lie algebras $\spl_{2}(\Fp)$ and $\gl_2(\Fp)$ (Theorem \ref{thm:sl2.fp} and \ref{thm:gl2.fp}).
	\item We enumerate subalgebras and ideals of the maximal class $\Fp$-Lie algebras $M_{c}$ for all $c\geq1$ (Theorem \ref{thm:fp.mc.ideal} and \ref{thm:fp.mc}).
	\item We enumerate ideals of the family of free class-$c$ nilpotent $\Fp$-Lie algebras on $d$ generators $\mff_{c,d}$ for  $\{(c,d)=(2,2),(3,2),(4,2)\}$ (Theorem \ref{thm:fc2}).
	\item We adapt Voll's method for computing ideal zeta functions of certain class-2 $\Fp$-Lie algebras (Lemma \ref{lem:idealc2}) and enumerate ideals of the Grenham ring $\mathcal{G}_{n}$ for $n\geq1$ and $\mff_{2,d}$ for $d\geq1$ (Theorem \ref{thm:fp.grenhamideal} and \ref{thm:fp.f2d}).
	\item We enumerate ideals of du Sautoy's elliptic curve example $L_{E}$ and other interesting nilpotent $\Fp$-Lie algebras connected to Higman's PORC conjecture (Theorem \ref{thm:Elliptic}, \ref{thm:Leenp8}, and \ref{thm:vl}).
	\item We compute ideal zeta functions for the family of solvable $\Fp$-Lie algebras $\textrm{tr}_{n}(\Fp)$, the set of $n\times n$ upper-triangular matrices,  for $1\leq n\leq4$ (Theorem \ref{thm:tr(n)}).
	\item We compute the $\Fp$ approximation of graded ideal zeta functions of 26 ``fundamental graded" Lie algebras of dimension at most 6, as classified in \cite{Kuzmich/99} (Section \ref{sec:5} Table \ref{tab:graded}).
\end{itemize}
Further questions will be discussed in Section \ref{sec:6}. 

Most of these computations have not been discussed in any literature before. To the author's knowledge, there is no known work in the literature that systematically investigated the zeta functions of $\Fp$-Lie algebras or $\Fq$-algebras in general. We hope this article provides a reasonable starting point for a more general theory.

\subsection{Applications and related works}
\subsubsection{Lie algebras over $\Fp$ and finite $p$-groups of exponent $p$}
The original motivation to introduce zeta functions of rings was to use them to study the subgroup growth problems in group theory. As the study of subgroup growth of finitely generated nilpotent groups can be linearized via the Mal'cev correspondence \cite[Theorem 4.1]{GSS/88}, one could translate it to the subring growth of suitable nilpotent Lie rings. For this reason most of the known results on zeta functions of rings have been focusing on nilpotent Lie rings. 

If we restrict $L(\Fp)$ to be a class-$c$ nilpotent  $\Fp$-Lie algebra of dimension $n$, then similar to the Lie ring case, for prime $p>c$ the Lazard correspondence gives $L(\Fp)$ a corresponding finite $p$-group $G(L(\Fp))$ of class-$c$, order $p^n$ and exponent $p$ using the Baker-Campbell-Hausdorff formula (\cite[Chapter 10]{Khukhro/98}). Like \cite[Theorem 4.1]{GSS/88}, \cite[Example 10.24]{Khukhro/98} suggests that for almost all primes $p>c$ we have
\[\zeta_{G(L(\Fp))}^{*}(s)=\zeta_{L(\Fp)}^{*}(s).\]
This implies that the method we use in this paper also allows us to count subgroups or normal subgroups of fixed-index in a given finite $p$-group of exponent $p$. For instance, all the zeta functions of nilpotent $\Fp$-Lie algebras $L(\Fp)$ in this article exhibits that
\begin{itemize}
	\item $a_{p^{i}}^{*}(L(\Fp))\neq0$ for all $0\leq i \leq n$,
	\item $a_{p}^{*}(L(\Fp))=\binom{r}{1}_p$, where $r$ is the number of generators of $L$,
\end{itemize}
which are compatible with the fact that for any finite $p$-group $G$ there are at least 1 normal subgroup of each index, and the number of maximal subgroups (hence normal and index $p$) is $\sum_{k=0}^{r-1}p^{k}$, where $r$ is the minimum size of a generating set for $G$.

\subsubsection{Applications on Higman's PORC conjecture}
Let $f_{n}(p)$ denote the number of isomorphism classes of $p$-groups of order $p^{n}$. In 1960, Higman \cite{HigmanII/60} conjectured that for a fixed $n$, there exists a fixed integer $N$ and finitely
many polynomials $g_{i}(x)\;(i=1,2,\ldots,N)$ such that if
$p\equiv i\bmod N$ then 
\[
f_{n}(p)=g_{i}(p).
\]
This is now known as Higman's PORC (Polynomial On Residue Classes) Conjecture. This conjecture is currently known to be true for $n\leq7$, but still remains open for $n\geq8$. For a detailed account with recent works see \cite{BEO/02}, \cite{VL/2012}, \cite{VL/2015}, or \cite{LVL22}.

Connection between the uniformity of ideal zeta functions of nilpotent Lie rings and Higman's PORC conjecture was first observed by du Sautoy, where he conjectured \cite[Conjecture 5.11]{duS/02} that proving the $\Zp$-uniformity of the ideal zeta functions $\zeta_{\mff_{c,d}}^{\triangleleft}(s)$ of $\mff_{c,d}$, a free class-$c$ nilpotent Lie ring on $d$ generators, for all $c,d\in\N$ might contribute to prove Higman's PORC conjecture. Grunewald, Segal, and Smith proved \cite{GSS/88} that $\zeta_{\mff_{2,d}}^{\triangleleft}(s)$ is always $\Zp$-uniform for all $d\in\N$. However, for $c\geq2$ we only know one explicit formula for $\zeta_{\mff_{3,2}}^{\triangleleft}(s)$, which also turned out to be $\Zp$-uniform.

After du Sautoy's discovery, several explicit  observations (e.g, \cite{DuSVL/2012,Lee/2016,Lee/20arxiv2}) have been made to support the connection between the $\Zp$-uniformity of zeta functions and Higman's PORC conjecture. In Section \ref{sec:4}, we compute the zeta functions of Lie rings connected to PORC conjecture over $\Fp$. Our new computations of  $\zeta_{L(\Fp)}^{*}(s)$ may provide better insight to understand the uniformity problem and this open conjecture.

\subsubsection{Zeta functions of submodules and graded submodules under endomorphisms}

Let $R$ be a commutative ring, $\mathcal{V}$ a finitely generated $R$-module, and $\Omega\subseteq \End_R(\mathcal{V})$ a set of $R$-endomorphisms of
$\mathcal{V}$. In \cite{Rossmann/16} Rossmann defined the (\emph{submodule}) \emph{zeta function of $\Omega$ acting on
	$\mathcal{V}$} to be 
$$\zeta_{\Omega \curvearrowright \mathcal{V}}(s) =
\sum_{\mathcal{U}\leq \mathcal{V}} |\mathcal{V}:\mathcal{U}|^{-s},$$
the Dirichlet generating series enumerating the
finite-index $\Omega$-invariant submodules $\mathcal{U}$ of
$\mathcal{V}$. If, in addition, we fix an $R$-module decomposition
$\mathcal{V} = \mathcal{V}_1 \oplus \dots \oplus \mathcal{V}_a$,  then the associated (\emph{graded
	submodule}) \emph{zeta function $\zeta^{\textup{gr}}_{\Omega
		\curvearrowright \mathcal{V}}(s)$}, is enumerating
\emph{graded}  $\Omega$-invariant submodules
of~$\mathcal{V}$.

For now, let $R=\Z$. It is easy to see that the ideal zeta function $\zeta_{L}^{\triangleleft}(s)$ is the zeta function of the \emph{adjoint representation} $\ad(L)\subseteq \End_\Z(L)$ acting on $L$:
\[\zeta_{L}^{\triangleleft}(s)=\zeta_{\ad(L)\curvearrowright L}(s).\]
We can analogously define the \emph{graded ideal zeta function of L} as
\[\zeta_{L}^{\triangleleft\gr}(s)=\zeta_{\ad(L)\curvearrowright L}^{\gr}(s).\] 

One can think of these graded ideal zeta functions as another approximations of $\zeta_{L}^{\triangleleft}(s)$. In \cite{LeeVoll/18}, Voll and the author studied the local graded ideal zeta functions of $\mff_{c,d}$, for some values of $c$ and $d$. In particular, we calculated the explicit formulas for $\zeta_{\mff_{3,3}(\Zp)}^{\idealgr}(s)$ and $\zeta_{\mff_{4,2}(\Zp)}^{\idealgr}(s)$, and studied their properties. Note that we still do not know the formulas for local (ungraded)  ideal zeta functions $\zeta_{\mff_{3,3}(\Zp)}^{\ideal}(s)$ and $\zeta_{\mff_{4,2}(\Zp)}^{\ideal}(s)$ yet.

Here one can also consider the graded ideal zeta functions of $\Fp$-Lie algebras, namely $\zeta_{L(\Fp)}^{\triangleleft\gr}(s)$. In this article we compute the graded ideal zeta functions of 26 fundamental graded $\Fp$-Lie algebras of dimension at most 6  (Section \ref{sec:5} Table \ref{tab:graded}).

\begin{rem}\label{rem:fp.lie}
	As we mentioned in Section 1.2 the arguments in this article actually work on general finite dimensional $\Fq$-algebras and their zeta functions as defined in \eqref{eq:fp.zeta} (see Remark \ref{rem:Zp.and.Fp}). However, to concentrate on the comparison among the zeta functions of Lie rings and Lie algebras over $\Z$, $\Zp$, and $\Fp$, and to stay compatible with the study of finite $p$-groups, we decided to concentrate on $\Fp$-Lie algebras.
\end{rem}

\subsection{Assumptions and Notations}
We write $\N=\{1,2,\dots\}$ for the set of all natural numbers and $\N_{0}$ for the set of all natural numbers including zero. Given $n\in\N$, we write $[n]$ for $\{1,2,\dots,n\}$.

Throughout this paper, unless stated otherwise, $L$ will denote a Lie ring additively isomorphic to $\Z^n$, $L(\Zp)\cong L\tensor\Z_p$, and $L(\Fp)\cong L\tensor\F_p$. By a class of Lie ring we always mean the nilpotent class. Whenever there is a presentation for a group or a ring, we always assume that up to anti-symmetry, all other unlisted commutators are trivial. We write $t:=p^{-s}$, where $s$ is a complex variable.

\section{Computing zeta functions of $\Fp$-Lie algebras}\label{sec:2}
Let $K$ be a field. The \textit{Grassmannian} $\Gr_{K}(n,m)$ is the set of all $m$-dimensional subspaces of $K^{n}$. Every element of $\Gr_{K}(n,m)$, an $m$-dimensional subspace, is a `point' of $\Gr_{K}(n,m)$ as a projective variety, and it can be described as a span of some $m$ independent row vectors of length $n$, in which we can arrange in a $m\times n$ matrix.

Let $L$ be a Lie ring and recall
\[\zeta_{L(\Fp)}^{*}(s)=\sum_{i=0}^{n}a_{p^{i}}^{*}(L(\Fp))t^{i}.\]
Since $L(\Fp)$ is now an $n$-dimensional vector space over $\Fp$, each coefficient $a_{p^{i}}^{*}(L(\Fp))$ enumerates the number of subspaces of codimension $i$ in $L(\Fp)\cong \Fp^{n}$ satisfying certain conditions to be subalgebras or ideals. Since being a subalgebra or an ideal is clearly a Zariski-closed condition, these coefficients enumerate the number of $\Fp$-points of a closed subvariety of a Grassmannian, which can be summarized as:
\begin{thm}\label{thm:fp.general}
There exist finitely many varieties $U_{1}^{*},\ldots,U_{h}^{*}$ defined over $\mathbb{Q},$  and polynomials $W_{1}^{*}(X,Y),\ldots,W_{h}^{*}(X,Y)\in\mathbb{Z}[X,Y]$ such that, for almost all primes $p$, 
\begin{equation*}
	\zeta_{L(\Fp)}^{*}(s)=\sum_{j=1}^{h}|\overline{U_{j}^{*}}(\Fp)|W_{j}^{*}(p,t).
\end{equation*} 
\end{thm}
However, like the study of (classical) zeta functions of Lie rings, the real question is whether one can explicitly find those varieties $U_{j}^{*}$ for a given $L(\Fp)$. The aim of this section is to provide a constructive method for this.

Let $\mathcal{B}=(e_1,\ldots,e_{n})$ be a basis of $L$. Normally, each $m$-dimensional subspace can be represented by a unique full-rank $m\times n$ matrix in Reduced Row Echelon Form (RREF) with respect to this basis. The subset of the Grassmannian where its elements have a particular RREF constitutes a \textit{Schubert Cell}, and $\Gr_{K}(n,m)$ is a disjoint union of Schubert Cells. To make computations easier, instead of a $m\times n$ matrix in a strict RREF, we choose to represent each subspace in a slightly modified way, by a unique form of $n\times n$ upper-triangular matrix of rank $m$ where each leading entry 1 in a non-zero row lies on the diagonal, as follows.
\begin{lem}\label{lem:matrix}
	Let $0_{i}$ denote an $i\times i$ zero matrix, $I_{j}$ an $j\times j$ identity matrix, and $D_{i,j}$ an arbitrary $i\times j$ matrix over $\mathbb{F}_{p}$. Then there is one-to-one correspondence between the subspace of $\LFp$ of index $p^{k}$ and the upper triangular matrix 
	\[M=(m_{i,j})_{i,j\in[n]}=\left(\begin{matrix}0_{a_1}&0&0&0&\cdots&0&0\\0&I_{b_1}&D_{b_1,a_2}&0&\cdots&D_{b_1,a_r}&0\\0&0&0_{a_2}&0&\cdots&0&0\\0&0&0&I_{b_2}&\cdots&D_{b_2,a_r}&0\\0&0&0&0&\ddots&0&0\\0&0&0&0&0&0_{a_r}&0\\0&0&0&0&0&0&I_{b_r}\end{matrix}\right)\in\textnormal{Tr}_n(\mathbb{F}_{p})\]
	of rank $n-k$, where  $\sum_{i=1}^{r}(a_{i})=k$,  $a_{1},b_{r}\geq0$, and $a_{2}\ldots,a_{r},b_{1},\ldots,b_{r-1}>0$.
\end{lem}

\begin{exm}
	Let $V$ be a 3-dimensional vector space over $\Fp$. In terms of Schubert cells, any 2-dimensional subspace $W$ of $V$ would be represented by
\begin{align*}
		\begin{pmatrix}
		1&0&a_{1,3}\\
		0&1&a_{2,3}
	\end{pmatrix},
		\begin{pmatrix}
	1&a_{1,2}&0\\
	0&0&1
\end{pmatrix},\textrm{ or }
		\begin{pmatrix}
	0&1&0\\
	0&0&1
\end{pmatrix},
\end{align*}
where $a_{1,2},a_{1,3},a_{2,3}\in\Fp$. In our setting, $W$ would be represented by
\begin{align*}
	\begin{pmatrix}
		1&0&a_{1,3}\\
		0&1&a_{2,3}\\
		0&0&0
	\end{pmatrix},
	\begin{pmatrix}
		1&a_{1,2}&0\\
		0&0&0\\
		0&0&1
	\end{pmatrix},\textrm{or}
	\begin{pmatrix}
		0&0&0\\
		0&1&0\\
		0&0&1
	\end{pmatrix}.
\end{align*} 
\end{exm}
Lemma \ref{lem:matrix} implies that for any given additive subgroup $M$ of index $p^k$ in $\LFp$, there are uniquely determined numbers $r\in\N$ and $\underline{a}=(a_{1},\ldots,a_{r}),\underline{b}=(b_{1},\ldots,b_{r})$ such that 
\[(m_{1,1},m_{2,2}\ldots,m_{n,n})=(0^{(a_{1})},1^{(b_{1})},\dots,0^{(a_{r})},1^{(b_{r})}).\] 

Let us call $(\underline{a},\underline{b})$ the \textit{diagonal type} of $M$, and write $d(M)=(\underline{a},\underline{b})$. Let $M_{\underline{a},\underline{b}}$ denote a matrix of diagonal type $(\underline{a},\underline{b})$,  and let $\mathcal{M}_{\underline{a},\underline{b}}$ denote the set of $n\times n$ upper triangular matrices over $\Fp$ with diagonal type $(\underline{a},\underline{b})$.

Like Schubert cells, let us call the subset of the Grassmannian where its elements have a particular diagonal type a \textit{diagonal cell}. So each $\mathcal{M}_{\underline{a},\underline{b}}$ is a diagonal cell of diagonal type $(\underline{a},\underline{b})$. Lemma \ref{lem:matrix} implies that the Grassmannian is a disjoint union of diagonal cells. Since each $a_{i},b_{i}$ can be either $0$ or $1$, there are $2^{n}$ diagonal cells of distinct diagonal types, and each of them satisfies $\left|\mathcal{M}_{\underline{a},\underline{b}}\right|=\sum_{i=1}^{r}\sum_{j=1}^{i-1}a_{i}b_{j}.$

\begin{exm} \label{exm:Heisenberg} Let 
	\[H=\langle x_{1},x_{2},x_{3}\,\mid\,[x_1,x_2]=x_3\rangle\]
	be the Heisenberg Lie ring. Then  $H(\Fp)$ is a 3-dimensional $\Fp$-Lie algebra of order $p^3$. The subspaces of  $H(\Fp)$ with index $p^k$ are precisely the following, represented by the diagonal cells:
	\begin{itemize}
		\item $\left(\begin{matrix}1&0&0\\&1&0\\&&1\\\end{matrix}\right)$ if $k=0$,
		\item $\left(\begin{matrix}0&0&0\\&1&0\\&&1\\\end{matrix}\right)$ or $\left(\begin{matrix}1&m_{1,2}&0\\&0&0\\&&1\\\end{matrix}\right)$ or
		$\left(\begin{matrix}1&0&m_{1,3}\\&1&m_{2,3}\\&&0\\\end{matrix}\right)$ if $k=1$,
		\item $\left(\begin{matrix}0&0&0\\&0&0\\&&1\\\end{matrix}\right)$ or $\left(\begin{matrix}0&0&0\\&1&m_{2,3}\\&&0\\\end{matrix}\right)$ or
		$\left(\begin{matrix}1&m_{1,2}&m_{1,3}\\&0&0\\&&0\\\end{matrix}\right)$ if $k=2$,
		\item $\left(\begin{matrix}0&0&0\\&0&0\\&&0\\\end{matrix}\right)$ if $k=3$.
	\end{itemize}
	Here we have $8=2^3$ distinct diagonal cells.
\end{exm}

Unfortunately, not every element of $\mathcal{M}_{\underline{a},\underline{b}}$ may give rise to subalgebras or ideals of $\LFp$. For a given $M_{\underline{a},\underline{b}}\in\mathcal{M}_{\underline{a},\underline{b}}$, we may consider the rows $\bsm_{1},\ldots,\bsm_{n}$ of this matrix to be additive generators of a subgroup of $L(\Fp)$. This will be a subalgebra of $L(\Fp)$ if
\begin{equation*}\label{eq:subalg.con}
	[\bsm_{i},\bsm_{j}]\in\langle\bsm_{1},\ldots,\bsm_{n}\rangle_{\Fp}\;\textrm{for all }1\leq i<j\leq n,
\end{equation*}
and an ideal if 
\begin{equation*}\label{eq:ideal.con}
	[\bsm_{i},\bold{e}_{j}]\in\langle\bsm_{1},\ldots,\bsm_{n}\rangle_{\Fp}\;\textrm{for all }1\leq i,j\leq n,
\end{equation*}
where the $n$-tuple $\bold{e}_{j}$ denotes the standard unit vector with 1 in the $j$th entry and zeros elsewhere. Hence we need to find out the conditions which make those matrices become subalgebras or ideals.

Let
\[F_{p}^{\leq}:=\left\{M\in\textnormal{Tr}(\mathbb{F}_{p})\,\mid \Fp^{n}\cdot M\leq L(\Fp)\right\}\]
and 
\[F_{p}^{\triangleleft}:=\left\{M\in\textnormal{Tr}(\mathbb{F}_{p})\,\mid \Fp^{n}\cdot M\ideal L(\Fp)\right\}.\]

Also let $F_{\underline{a},\underline{b}}^{\leq}=F_{p}^{\leq}\cap \mathcal{M}_{\underline{a},\underline{b}}$ and 
$F_{\underline{a},\underline{b}}^{\triangleleft}=F_{p}^{\triangleleft}\cap \mathcal{M}_{\underline{a},\underline{b}}.$ Then we have 
\[\zeta_{\LFp}^{\leq}(s)=\sum_{\underline{a},\underline{b}}\left|F_{\underline{a},\underline{b}}^{\leq}\right|p^{-(\sum_{i=1}^{r}a_{i})s}\]
and
\[\zeta_{\LFp}^{\triangleleft}(s)=\sum_{\underline{a},\underline{b}}\left|F_{\underline{a},\underline{b}}^{\triangleleft}\right|p^{-(\sum_{i=1}^{r}a_{i})s}.\]

For a given \[M_{\underline{a},\underline{b}}=\left(\begin{matrix}0_{a_1}&0&0&0&\cdots&0&0\\0&I_{b_1}&D_{b_1,a_2}&0&\cdots&D_{b_1,a_r}&0\\0&0&0_{a_2}&0&\cdots&0&0\\0&0&0&I_{b_2}&\cdots&D_{b_2,a_r}&0\\0&0&0&0&\ddots&0&0\\0&0&0&0&0&0_{a_r}&0\\0&0&0&0&0&0&I_{b_r}\end{matrix}\right),\]
let
\[M^{\sharp}_{\underline{a},\underline{b}}=\left(\begin{matrix}I_{a_1}&0&0&0&\cdots&0&0\\0&I_{b_1}&-D_{b_1,a_2}&0&\cdots&-D_{b_1,a_r}&0\\0&0&I_{a_2}&0&\cdots&0&0\\0&0&0&I_{b_2}&\cdots&-D_{b_2,a_r}&0\\0&0&0&0&\ddots&0&0\\0&0&0&0&0&I_{a_r}&0\\0&0&0&0&0&0&I_{b_r}\end{matrix}\right),\]
and
\[M^{\flat}_{\underline{a},\underline{b}}=\left(\begin{matrix}0_{a_1}&0&0&0&\cdots&0&0\\0&I_{b_1}&0&0&\cdots&0&0\\0&0&0_{a_2}&0&\cdots&0&0\\0&0&0&I_{b_2}&\cdots&0&0\\0&0&0&0&\ddots&0&0\\0&0&0&0&0&0_{a_r}&0\\0&0&0&0&0&0&I_{b_r}\end{matrix}\right).\]
One can directly check the following result.
\begin{lem}\label{lem:REF}
	For given $r\in\mathbb{N}$, $\underline{a}$ and $\underline{b}$, we have
	\[M_{\underline{a},\underline{b}}M^{\sharp}_{\underline{a},\underline{b}}=M^{\flat}_{\underline{a},\underline{b}}.\]
\end{lem}
 With this we can compute $\left|F_{\underline{a},\underline{b}}^{\leq}\right|$ and $\left|F_{\underline{a},\underline{b}}^{\triangleleft}\right|$. For $i,j\in[n]$, let $C_j$ denote the matrix whose rows are $\bold{c}_{i}=[\bold{e}_{i},\bold{e}_{j}]$. We have $M_{\underline{a},\underline{b}}\in F_{\underline{a},\underline{b}}^{\leq}$ if and only if we can solve for each $1\leq i,j\leq n$ the equation
 \begin{equation}\label{eq:subalg} 
\bold{m}_{i}\left(\sum_{l=j}^{n}m_{j,l}C_{l}\right)=\left(Y^{(1)}_{i,j},\ldots,Y^{(n)}_{i,j}\right)M_{\underline{a},\underline{b}}
\end{equation}
	with $\left(Y^{(1)}_{i,j},\ldots,Y^{(n)}_{i,j}\right)\in\mathbb{Z}^{n}.$ Similarly, $M_{\underline{a},\underline{b}}\in F_{\underline{a},\underline{b}}^{\triangleleft}$ if and only if we can solve for each $1\leq i,j\leq n$ the equation 
	\begin{equation}\label{eq:ideal} 
		\bold{m}_{i}C_{j}=\left(Y^{(1)}_{i,j},\ldots,Y^{(n)}_{i,j}\right)M_{\underline{a},\underline{b}}
	\end{equation}
	with $\left(Y^{(1)}_{i,j},\ldots,Y^{(n)}_{i,j}\right)\in\mathbb{Z}^{n}.$
	
	Let us consider the ideals first. By Lemma \ref{lem:REF} we can rewrite \eqref{eq:ideal} as 
	\[\bold{m}_{i}C_{j}M_{\underline{a},\underline{b}}^{\sharp}=\left(Y^{(1)}_{i,j},\ldots,Y^{(n)}_{i,j}\right)M^{\flat}_{\underline{a},\underline{b}}.\]
	Now, let $g_{ijk}^{\triangleleft}(m_{r,s})$ denote the $k$-th entry of the $n$-tuple $\bold{m}_{i}C_{j}M^{\sharp}_{\underline{a},\underline{b}}$. The way we defined our $M$ implies that $g_{ijk}^{\triangleleft}(m_{r,s})$ is at most degree-2 polynomial in $m_{r,s}$. Since $M^{\flat}_{\underline{a},\underline{b}}=\textrm{diag}(0^{(a_{1})},1^{(b_{1})},\dots,0^{(a_{r})},1^{(b_{r})})$, it is clear that the $k$-th entry of the $n$-tuple $\left(Y^{(1)}_{i,j},\ldots,Y^{(n)}_{i,j}\right)M^{\flat}_{\underline{a},\underline{b}}$ is either 0 or $Y^{(k)}_{i,j}$. Since the condition $g_{ijk}^{\triangleleft}(m_{r,s})=Y^{(k)}_{i,j}$ is redundant, we can conclude that $M_{\underline{a},\underline{b}}$ gives an ideal if and only if we can solve a system of equations $g_{ijk}^{\triangleleft}(m_{r,s})=0$ for $1\leq i,j,k\leq n$ where $m_{k,k}=0$.
	
	Finally, let $\bold{Y}_{\underline{a},\underline{b}}^{\triangleleft}$ denote the variety defined by the corresponding system of equations $g_{ijk}^{\triangleleft}(m_{r,s})=m_{k,k}Y^{(k)}_{i,j}$, and let $c_{\underline{a},\underline{b}}^{\triangleleft}=\mid\bold{Y}_{\underline{a},\underline{b}}^{\triangleleft}(\mathbb{F}_{p})\mid$. Then we get $\left|F_{\underline{a},\underline{b}}^{\triangleleft}\right|=c_{\underline{a},\underline{b}}^{\triangleleft}$.
	
	Hence \[\zeta_{\LFp}^{\triangleleft}(s)=\sum_{\underline{a},\underline{b}}c_{\underline{a},\underline{b}}^{\triangleleft}\cdot p^{-(\sum_{i=1}^{r}a_{i})s}.\]
	
	For subalgebras, one can similarly define $g_{ijk}^{\leq}(m_{rs})$, $\bold{Y}_{\underline{a},\underline{b}}^{\leq}$ and $c_{\underline{a},\underline{b}}^{\leq}$ from \eqref{eq:subalg} and see that 
	\[\zeta_{\LFp}^{\leq}(s)=\sum_{\underline{a},\underline{b}}c_{\underline{a},\underline{b}}^{\leq}\cdot p^{-(\sum_{i=1}^{r}a_{i})s}.\]
We demonstrate how this method works with the following example:
\begin{thm} \label{thm:H.theory} Let $H$ denote the Heisenberg Lie ring as in Example \ref{exm:Heisenberg}. We have
	\begin{align*}
		\zeta_{H(\Fp)}^{\triangleleft}(s)&=1+(1+p)t+t^{2}+t^{3},\\
		\zeta_{H(\Fp)}^{\leq}(s)&=1+(1+p)t+(1+p+p^{2})t^{2}+t^{3}.
	\end{align*}
\end{thm}
\begin{proof}
	Recall Example \ref{exm:Heisenberg}.  There are 8 distinct diagonal cells corresponding to the subspaces of $H(\Fp)$ of index $p^{k}$ for $0\leq k \leq 3$. We can explicitly calculate $\zeta_{\LFp}^{\triangleleft}(s)$ using these diagonal cells.
	
	Start with $\left(\begin{matrix}1&0&0\\&1&0\\&&1\\\end{matrix}\right)$ and $\left(\begin{matrix}0&0&0\\&0&0\\&&0\\\end{matrix}\right)$. They trivially give 1 ideal of index 1 and $p^3$ respectively. 
	
	Consider $M_{\underline{a},\underline{b}}=\left(\begin{matrix}0&0&0\\&1&0\\&&1\\\end{matrix}\right)$. Then our equation 
	\[\bold{m}_{i}C_{j}M_{\underline{a},\underline{b}}^{\sharp}=\left(Y^{(1)}_{i,j},\ldots,Y^{(n)}_{i,j}\right)M^{\flat}_{\underline{a},\underline{b}}\]
	gives rise to the system of equations 
	\begin{align*}
		i=1,j=1:\,&(0,0,0)=(0,Y_{1,1}^{(2)},Y_{1,1}^{(3)}),\\
		i=1,j=2:\,&(0,0,0)=(0,Y_{1,2}^{(2)},Y_{1,2}^{(3)}),\\
		i=1,j=3:\,&(0,0,0)=(0,Y_{1,3}^{(2)},Y_{1,3}^{(3)}),\\
		i=2,j=1:\,&(0,0,-1)=(0,Y_{2,1}^{(2)},Y_{2,1}^{(3)}),\\
		i=2,j=2:\,&(0,0,0)=(0,Y_{2,2}^{(2)},Y_{2,2}^{(3)}),\\
		i=2,j=3:\,&(0,0,0)=(0,Y_{2,3}^{(2)},Y_{2,3}^{(3)}),\\
		i=3,j=1:\,&(0,0,0)=(0,Y_{3,1}^{(2)},Y_{3,1}^{(3)}),\\
		i=3,j=2:\,&(0,0,0)=(0,Y_{3,2}^{(2)},Y_{3,2}^{(3)}),\\
		i=3,j=3:\,&(0,0,0)=(0,Y_{3,3}^{(2)},Y_{3,3}^{(3)}).
	\end{align*}
	Since we can solve this for all $i,j\in[3]$, we get $c_{\underline{a},\underline{b}}^{\triangleleft}=1.$ Similarly, for  $M_{\underline{a},\underline{b}}=\left(\begin{matrix}1&m_{1,2}&0\\&0&0\\&&1\\\end{matrix}\right)$, the same method gives $c_{\underline{a},\underline{b}}^{\triangleleft}=p.$
	
	The interesting case is when $M_{\underline{a},\underline{b}}=\left(\begin{matrix}1&0&m_{1,3}\\&1&m_{2,3}\\&&0\\\end{matrix}\right)$. For $i=1, j=2$, we get
	
	\[i=1,j=2:(0,0,1)=(Y_{1,2}^{(1)},Y_{1,2}^{(2)},0),\]	
	and one needs to solve $1=0$, which is impossible. Hence we get $c_{\underline{a},\underline{b}}^{\triangleleft}=0,$ and all together we get $1+p$ ideals of index $p$. 
	
	Similarly, $\left(\begin{matrix}0&0&0\\&0&0\\&&1\\\end{matrix}\right)$ gives 1 ideal of index $p^2$, and
	$\left(\begin{matrix}0&0&0\\&1&m_{2,3}\\&&0\\\end{matrix}\right)$ and
	$\left(\begin{matrix}1&m_{1,2}&m_{1,3}\\&0&0\\&&0\\\end{matrix}\right)$ give
	no ideals. 
	
	Summing all up, we get $\zeta_{H(\Fp)}^{\triangleleft}(s)=1+(1+p)t+t^{2}+t^{3}$ as required. The same method gives $\zeta_{H(\Fp)}^{\leq}(s)=1+(1+p)t+(1+p+p^{2})t^{2}+t^{3}$. 
\end{proof}
\begin{rem}\label{rem:fq.gen}
	Note that in fact the same argument works for general finite fields $\Fq$, and our results can be extended from $\Fp$ to $\Fq$ simply by writing
	\[\zeta_{L(\Fq)}^{*}(s)=\sum_{\underline{a},\underline{b}}\mid\bold{Y}_{\underline{a},\underline{b}}^{*}(\mathbb{F}_{q})\mid\cdot q^{-(\sum_{i=1}^{r}a_{i})s}.\]
	For example, one can extend Theorem \ref{thm:H.theory} as 
\begin{align*}
	\zeta_{H(\Fq)}^{\triangleleft}(s)&=1+(1+q)t+t^{2}+t^{3},\\
	\zeta_{H(\Fq)}^{\leq}(s)&=1+(1+q)t+(1+q+q^{2})t^{2}+t^{3},
\end{align*}
and	Theorem \ref{thm:Elliptic} as	
	\begin{align*}\zeta_{L_{E}(\Fq)}^{\triangleleft}(s) =& \zeta_{\mathbb{F}_{q}^{6}}(s)+q^{2}\left|E(\mathbb{F}_{q})\right|t^{5}+\left|E(\mathbb{F}_{q})\right|\left(q+q^{2}\right)t^{6}\\
		&+\binom{3}{1}_{q}t^{7} +\binom{3}{2}_{q}t^{8}+t^{9},
	\end{align*}
	where now $t=q^{-s}$. 
\end{rem}
\begin{rem}\label{rem:Zp.and.Fp} 
This constructive method for computing $\zeta_{L(\Fp)}^{*}(s)$ is very similar to the argument used by du Sautoy and Grunewald to prove \cite[Theorem 1.3]{duSG/00}. In particular, we explicitly described varieties $U_{j}^{*}$ in Theorem \ref{thm:fp.general} as $\bold{Y}_{\underline{a},\underline{b}}^{*}$. Since one of the main motivations to study $\zeta_{L(\Fp)}^{*}(s)$ is to use it as an approximation of $\zeta_{L(\Zp)}^{*}(s)$, it is very natural to ask whether the set of algebraic varieties coming up from $\zeta_{L(\Fp)}^{*}(s)$ is the subset of that for $\zeta_{L(\Zp)}^{*}(s)$. In other words, one can ask whether the varieties we need to understand to compute $\zeta_{L(\Fp)}^{*}(s)$ are also among those necessary for $\zeta_{L(\Zp)}^{*}(s)$. Although we cannot rigorously prove this in this article yet, our computations suggest this in a very favorable way. More computations of $\zeta_{L(\Fp)}^{*}(s)$ for further investigations would be desired.
\end{rem}
\begin{rem}
	One potential way to tackle this problem might be appealing to motivic zeta functions introduced in \cite{duSLoeser/04}. For example, is the subring of the Grothendieck ring of varieties generated by the varieties coming up in $\zeta_{L(\Zp)}^{*}(s)$ the same as the one generated by the varieties coming up in $\zeta_{L(\Fp)}^{*}(s)$? Unfortunately the author cannot answer this problem in this article, but will further investigate in this direction.
\end{rem}

\subsection{Applications on semisimple cases: $\spl_{2}$ and $\gl_2$}\label{subsec:semi}
As we mentioned in the introduction, one can also consider the subalgebra or ideal zeta functions of non-nilpotent Lie algebras. For example, the generic local subalgebra zeta function of semisimple Lie algebra $\spl_2(\Z)$ was computed several times by various mathematicians using different methods (e.g., \cite{duS/00sl2},\cite{duSTaylor/02}, and \cite{KlopschVoll/09MZ}). However, this was the only example of a semisimple Lie algebras whose zeta functions have been computed, until Rossmann computed the local subalgebra zeta function of $\gl_2(\Z)$ in \cite{Rossmann/16}.

Here we compute $\zeta_{\spl_{2}(\Fp)}^{\leq}(s)$ and $\zeta_{\gl_{2}(\Fp)}^{\leq}(s)$ and show how the method we developed also works for semisimple Lie algebras in the same way.

\begin{thm}\label{thm:sl2.fp} Let \[\spl_{2}=\langle e,f,h\mid[h,e]=2e,[h,f]=-2f,[e,f]=h\rangle\] denote a semisimple Lie ring isomorphic to $\Z^{3}$. Then we have	\[\zeta_{\spl_{2}(\Fp)}^{\leq}(s)=1+(1+p)t+(1+p+p^{2})t^{2}+t^{3}.\]
\end{thm}

\begin{proof}
	Let 
	\[
	M=\left(\begin{array}{ccc}
		m_{1,1} & m_{1,2} & m_{1,3} \\
		0 & m_{2,2} & m_{2,3} \\
		0 & 0 & m_{3,3} 
	\end{array}\right)
	\]
	denote the matrix representations of subspaces of $\spl_{2}(\Fp)$. Since $[\bsm_{i},\bsm_{i}]=0$ for all $i\in[3]$, one can easily see that
	\begin{itemize}
		\item $\left(\begin{matrix}1&0&0\\&1&0\\&&1\\\end{matrix}\right)$ gives $1$,
		\item $\left(\begin{matrix}0&0&0\\&0&0\\&&1\\\end{matrix}\right)$,  $\left(\begin{matrix}0&0&0\\&1&m_{2,3}\\&&0\\\end{matrix}\right)$, and
		$\left(\begin{matrix}1&m_{1,2}&m_{1,3}\\&0&0\\&&0\\\end{matrix}\right)$ give $\binom{3}{1}_pt^2$, and
		\item $\left(\begin{matrix}0&0&0\\&0&0\\&&0\\\end{matrix}\right)$ gives $t^3$.
	\end{itemize}
	Let us now consider the remaining cases. For	
	\[M=\begin{pmatrix}
		1&0&m_{1,3}\\
		&1&m_{2,3}\\
		&&0		
	\end{pmatrix},\]
	we need to solve $1+4m_{13}m_{2,3}=0$, giving $(p-1)t$. For	
	\[M=\begin{pmatrix}
		1&m_{1,2}&0\\
		&0&0\\
		&&1		
	\end{pmatrix},\]
	we need to solve $m_{1,2}=0$, giving $t$. Finally, for 
	\[M=\begin{pmatrix}
		0&0&0\\
		&1&0\\
		&&1		
	\end{pmatrix},\]
	it is a subalgebra of $\spl_2(\Fp)$, giving another $t$. Therefore we have\[\zeta_{\spl_{2}(\Fp)}^{\leq}(s)=1+(1+p)t+(1+p+p^2)t^{2}+t^{3}\]
	as required.
\end{proof}
Now, for $\gl_{2}$, as Rossmann did in \cite[Theorem 9.1]{Rossmann/16} we use the fact that $\gl_2(\Fp)$ is isomorphic to $\spl_2(\Fp)\oplus \Fp$ for $p\neq2$, where we regard $\Fp$ as an abelian $\Fp$-Lie algebra.
\begin{thm}\label{thm:gl2.fp} Let \[\gl_{2}=\langle e,f,h,g\mid[h,e]=2e,[h,f]=-2f,[e,f]=h\rangle\] denote a semisimple Lie ring isomorphic to $\Z^{4}$. Then for $p>2$ we have
	\begin{align*}
		\zeta_{\gl_2(\Fp)}^{\leq}(s)&=\zeta_{\spl_2(\Fp)}^{\leq}(s)+t+(p+p^2)t^2+(p+p^2+p^3)t^3+t^4\\
		&=	1+(2+p)t+(1+2p+2p^{2})t^{2}+(1+p+p^2+p^3)t^{3}+t^4
	\end{align*}
\end{thm}

\begin{proof}
	Let 
	\[
	M=\begin{pmatrix}
		m_{1,1}&m_{1,2}&m_{1,3}&m_{1,4}\\
		&m_{2,2}&m_{2,3}&m_{2,4}\\
		&&m_{3,3}&m_{3,4}\\
		&&&m_{4,4}		
	\end{pmatrix}
	\]
	denote the matrix representations of subspaces of $\gl_{2}(\Fp)$. 
	First, suppose $m_{4,4}=1$.	By the way we constructed $\gl_2(\Fp)$ as $\spl_2(\Fp)\oplus \Fp$  with the basis $\langle e,f,h,g \rangle$, from this case we get $\zeta_{\gl_{2}(\Fp)}^{\leq}(s)$ as computed in Theorem \ref{thm:sl2.fp}.
	
	Hence suppose $m_{4,4}=0$. Again one first see that 
	\begin{itemize}
		\item
		$\begin{pmatrix}
			0&0&0&0\\
			&0&0&0\\
			&&0&0\\
			&&&0		
		\end{pmatrix}$	
		gives $t^4$, and
		\item
		$\begin{pmatrix}
			1&m_{1,2}&m_{1,3}&m_{1,4}\\
			&0&0&0\\
			&&0&0\\
			&&&0		
		\end{pmatrix}$,
		$\begin{pmatrix}
			0&0&0&0\\
			&1&m_{2,3}&m_{2,4}\\
			&&0&0\\
			&&&0		
		\end{pmatrix}$, and
		$\begin{pmatrix}
			0&0&0&0\\
			&0&0&0\\
			&&1&m_{3,4}\\
			&&&0		
		\end{pmatrix}$ give $(p+p^2+p^3)t^3$.
	\end{itemize}
	Hence we check the remaining 4 cases. For
	\[M=\begin{pmatrix}
		1&0&m_{1,3}&m_{1,4}\\
		&1&m_{2,3}&m_{3,4}\\
		&&0&0\\
		&&&0		
	\end{pmatrix},\]
	we need to solve $1+4m_{1,3}m_{1,4}=m_{1,4}m_{2,3}+m_{1,3}m_{2,4}=0$. Since the first equation implies $m_{2,3}\neq0$, we have $p-1$ choices for $m_{2,3}$, $p$ choices for $m_{1,4}\in\Fp$, giving $m_{1,3}=-\frac{1}{4m_{2,3}}$ and $m_{2,4}=4m_{1,3}m_{2,3}^2$. Hence we get $p(p-1)t^2$. For
	\[M=\begin{pmatrix}
		1&m_{1,2}&0&m_{1,4}\\
		&0&0&0\\
		&&1&m_{3,4}\\
		&&&0		
	\end{pmatrix},\]
	we have $m_{1,2}=m_{1,4}=0$, giving $pt^2$. For
	\[M=\begin{pmatrix}
		0&0&0&0\\
		&1&0&m_{2,4}\\
		&&1&m_{3,4}\\
		&&&0		
	\end{pmatrix},\]
	we have $m_{2,4}=0$, giving another $pt^2$. Finally, for	
	\[M=\begin{pmatrix}
		1&0&0&m_{1,4}\\
		&1&0&m_{2,4}\\
		&&1&m_{3,4}\\
		&&&0		
	\end{pmatrix},\]  
	we have $m_{1,4}=m_{2,4}=m_{3,4}=0$, giving $t$. Therefore we have	
	\begin{align*}
		\zeta_{\gl_2(\Fp)}^{\leq}(s)&=\zeta_{\spl_2(\Fp)}^{\leq}(s)+t+(p+p^2)t^2+(p+p^2+p^3)t^3+t^4\\
		&=	1+(2+p)t+(1+2p+2p^{2})t^{2}+(1+p+p^2+p^3)t^{3}+t^4
	\end{align*}
	as required.
\end{proof}

\begin{rem}
	Note that $\spl_2$ and $\gl_2$ are currently the only semisimple Lie algebras whose local subalgebra zeta functions are known. One might try to compute $\zeta_{\spl_3(\Fp)}^{\leq}(s)$ or $\zeta_{\gl_3(\Fp)}^{\leq}(s)$  first before computing $\zeta_{\spl_3(\Zp)}^{\leq}(s)$ or $\zeta_{\gl_3(\Zp)}^{\leq}(s)$.	
\end{rem}

\section{Computations in practice for various $\Fp$-Lie algebras}\label{sec:3}

The method we developed in Section \ref{sec:2} can be summarized as follows:
\begin{enumerate}
	\item For a given $L(\Fp)$, write down $2^{n}$ distinct diagonal cells $\mathcal{M}_{\underline{a},\underline{b}}$.
	\item For each  $\mathcal{M}_{\underline{a},\underline{b}}$, construct $\bold{Y}_{\underline{a},\underline{b}}^{*}$ and solve $c_{\underline{a},\underline{b}}^{*}$ .
	\item Add together to get  $\zeta_{\LFp}^{*}(s)=\sum_{\underline{a},\underline{b}}c_{\underline{a},\underline{b}}^{*}\cdot p^{-(\sum_{i=1}^{r}a_{i})s}.$
\end{enumerate}
 This worked well with the Heisenberg case in Theorem \ref{thm:H.theory} as well as with the semisimple cases $\spl_2$ and $\gl_2$ in Theorem \ref{thm:sl2.fp} and \ref{thm:gl2.fp}. But it also has clear limitations. First, the number of diagonal cells we need to check grows exponentially.  Also, in general Step (2), solving certain systems of equations over $\Fp$, is very difficult. 

Fortunately, by putting some restrictions on $L$ we can apply additional techniques to make this computation easier in practice. In Section \ref{subsec:nil.fp}, we demonstrate how one can modify our general method for practical computations when $L$ is nilpotent. In Section \ref{subsec:c2}, by appealing to a well-known idea developed by Voll (\cite{Voll/04,Voll/05,Voll/11}), we show an alternative method for computing $\zeta_{L(\Fp)}^{\triangleleft}(s)$ when $L$ is class-2 nilpotent. 

\subsection{Nilpotent $\Fp$-Lie algebras}\label{subsec:nil.fp}
For nilpotent $L$, we can adopt following techniques:
\begin{enumerate}
	\item \textbf{Choice of the basis and the lower central series of $L$}: In Theorem \ref{thm:H.theory}, we chose our basis $\mathcal{B}=(e_{1},e_{2},e_{3})$ such that $\langle e_{3}\rangle=[H,H]=:H'$ and $\langle e_{1},e_{2}\rangle=H/H'$. In general, let $c$ be the nilpotency class of $\LFp$, $(\gamma_{i}(L))_{i=1}^{c}$ be the lower central series of $\LFp$, and  for $i=1,\dots,c$ set
	$\LFp^{(i)}:=\gamma_{i}(\LFp)/\gamma_{i+1}(\LFp)$.
	Let $d_{i}=\dim(\LFp^{(i)})$. To make computations ``simpler'', one can choose the basis such that the first $d_{1}$ elements generate $\gamma_{1}(\LFp)/\gamma_{2}(\LFp)$, the next $d_{2}$ generate $\gamma_{2}(\LFp)/\gamma_{3}(\LFp)$, and so on.
	\item \textbf{Diagonal conditions}: Among these conditions there are some simple ones of the form $m_{i,i}|m_{j,j}$, one diagonal entry dividing the other diagonal entry for ideals, and of the form $m_{i,i}|m_{j,j}m_{k,k}$ for subalgebras. We call them \textit{diagonal conditions}. One can think of these conditions as the backbone of the structure of subalgebras or ideals of $L$.  
\end{enumerate}
In particular, as we see shortly (e.g. Theorem \ref{thm:fp.mc.ideal} and Theorem \ref{thm:fp.mc}), this method even allows us to compute the zeta functions of some families of infinitely many $\Fp$-Lie algebras. 

\subsubsection{The Heisenberg Lie ring $H$, revisited}
Let 
\[
M=\left(\begin{array}{ccc}
	m_{1,1} & m_{1,2} & m_{1,3} \\
	0 & m_{2,2} & m_{2,3} \\
	0 & 0 & m_{3,3} 
\end{array}\right)
\]
denote the matrix representations of subspaces of $H$ as before. One can easily check that $M\triangleleft H$ if and only if $m_{3,3}|m_{1,1},m_{2,2}$. Hence $m_{3,3}=0$ forces $m_{1,1}=m_{2,2}=0$, and we only need to check 5 cases, namely
\[\left(\begin{matrix}1&0&0\\&1&0\\&&1\\\end{matrix}\right),\,\left(\begin{matrix}0&0&0\\&1&0\\&&1\\\end{matrix}\right),\,\left(\begin{matrix}1&m_{1,2}&0\\&0&0\\&&1\\\end{matrix}\right),\,\left(\begin{matrix}0&0&0\\&0&0\\&&1\\\end{matrix}\right),\,\textrm{and}\,\left(\begin{matrix}0&0&0\\&0&0\\&&0\\\end{matrix}\right).\]

This also explains why the systems of equations of the other three cases do not have solutions. For example, $\left(\begin{matrix}1&0&m_{1,3}\\&1&m_{2,3}\\&&0\\\end{matrix}\right)$ asked us to solve
\[i=1,j=2:\,(0,0,1)=(Y_{1,2}^{(1)},Y_{1,2}^{(2)},0),\]
which required $1=0$. 

\subsubsection{The maximal class-$c$ Lie algebra $M_{c}$, counting ideals}
For $c\geq2$, let \[M_{c}=\langle x_{0},x_{1},\ldots,x_{c}\mid\forall 1\leq i\leq 
c-1,[x_{0},x_{i}]=x_{i+1}\rangle\] be a maximal class-$c$ Lie ring $M_{c}$. 
\begin{thm} \label{thm:fp.mc.ideal}
	For $c\geq2$,
	\begin{align*}
		\zeta_{M_{c}(\Fp)}^{\triangleleft}(s)&=1+(1+p)t+\sum_{k=2}^{c+1}t^{k}\\
		&=\zeta_{\Fp^{2}}(s)+\sum_{k=3}^{c+1}t^{k}.
	\end{align*}
\end{thm}
\begin{proof}
	Note that when $c=2$, $M_{2}=H$ and we have \[\zeta_{M_{2}(\Fp)}^{\triangleleft}(s)=\zeta_{H(\Fp)}^{\triangleleft}(s)=1+(1+p)t+t^{2}+t^{3}\]
	as required. So suppose $c\geq3$.
	Let 
	\[
	M=\left(\begin{array}{cccc}
		m_{1,1} & m_{1,2} & \cdots & m_{1,c+1}\\
		0 & m_{2,2} & \cdots & m_{2,c+1}\\
		0 & 0 & \ddots & \vdots\\
		0 & 0 & \cdots & m_{c+1,c+1}
	\end{array}\right)
	\]
	denote a generator matrix for a subspace of $M_{c}(\Fp),$ and for $1\leq i \leq c+1$, let 
	$\boldsymbol{m}_{i}$ denote the $i$-th row of $M$. Since $[e_{1},e_{i}]=e_{i+1}$ for all $2\leq i \leq c$ and $[e_{i},e_{j}]=0$ for all $2\leq i,j \leq c+1$, we only need to consider
	\begin{align*}
		[\boldsymbol{m}_{1},e_{1}]&=-(m_{1,2}e_{3}+m_{1,3}e_{4}+\cdots+m_{1,c}e_{c+1}),\\
		[\boldsymbol{m}_{i},e_{1}]&=-(m_{i,i}e_{i+1}+m_{i,i+1}e_{i+2}+\cdots+m_{i,c}e_{c+1}),\\
		[\boldsymbol{m}_{1},e_{j}]&=m_{1,1}e_{j+1},	
	\end{align*}
	where all other commutators are trivial. This implies in particular that for $3\leq i\leq c+1,$ we have diagonal conditions $m_{i,i}\mid m_{j,j}$ for all
	$1\leq j<i.$ Now, let $m_{3,3}=1.$ Then it forces $m_{4,4}=\cdots=m_{c+1,c+1}=1$. Hence there are only
	4 such possible cases with $m_{3,3}=1$:
	
	\[
	\begin{aligned}M_{111\ldots1}=\left(\begin{array}{ccccc}
			1 & 0 & 0 & 0 & 0\\
			0 & 1 & 0 & 0 & 0\\
			0 & 0 & 1 & \cdots & 0\\
			0 & 0 & \vdots & \ddots & \vdots\\
			0 & 0 & 0 & \cdots & 1
		\end{array}\right),\, &	M_{011\ldots1}=\left(\begin{array}{ccccc}
			0 & 0 & 0 & 0 & 0\\
			0 & 1 & 0 & 0 & 0\\
			0 & 0 & 1 & \cdots & 0\\
			0 & 0 & \vdots & \ddots & \vdots\\
			0 & 0 & 0 & \cdots & 1
		\end{array}\right),\\
		M_{101\ldots1}=\left(\begin{array}{ccccc}
			1 & m_{1,2} & 0 & 0 & 0\\
			0 & 0 & 0 & 0 & 0\\
			0 & 0 & 1 & \cdots & 0\\
			0 & 0 & \vdots & \ddots & \vdots\\
			0 & 0 & 0 & \cdots & 1
		\end{array}\right),\,& 
		M_{001\ldots1}=\left(\begin{array}{ccccc}
			0 & 0 & 0 & 0 & 0\\
			0 & 0 & 0 & 0 & 0\\
			0 & 0 & 1 & \cdots & 0\\
			0 & 0 & \vdots & \ddots & \vdots\\
			0 & 0 & 0 & \cdots & 1
		\end{array}\right).
	\end{aligned}
	\]
	This gives us one ideal of index 1, a total of $p+1$ ideals of index $p$ and
	one ideal of index $p^{2}.$ For the rest, $M$ would look like
	\[
	M_{\underbrace{0\cdots 0}_{a_{1}}\underbrace{1\cdots1}_{b_{1}}}=\left(\begin{array}{cc}
		0_{a_{1}} & 0\\
		0 & I_{b_{1}}
	\end{array}\right)
	\]
	where $3\leq a_{1}\leq c+1$.
	Each of them gives 1 ideal of index $p^{a_{1}}.$ Thus we get 
	\begin{align*}
		\zeta_{M_{c}(\Fp)}^{\triangleleft}(s)&=1+(1+p)t+t^2+\sum_{k=3}^{c+1}t^{k}\\
		&=\zeta_{\Fp^{2}}(s)+\sum_{k=3}^{c+1}t^{k}.
	\end{align*}	
\end{proof}
Note that one only needs to check $c+3$ cases instead of $2^{c+1}$ cases. For $p>c$, it also gives an alternate proof of a well-known fact (\cite[p.9]{Blackburn58} or \cite[Section 9]{Berk08}) from finite $p$-group theory that any maximal class-$c$ group of order $p^{c+1}$ contains only one normal subgroup of order $p^{m}$ for $1\leq m\leq c-1$ and $1+p$ maximal (hence normal) subgroups or order $p^{c}$. 

\begin{rem} The local ideal zeta function $\zeta_{M_{c}(\Zp)}^{\triangleleft}(s)$	is extremely difficult to calculate. At present we only know the explicit formulas for $c\leq4$, first computed in \cite{Taylor/01}.
\end{rem}

We also computed the following result using similar computations:
\begin{cor}\label{cor:fil4}
	Let 
	\[
	\text{Fil}_{4}=\langle x_{1},\,x_{2},\,x_{3},\,x_{4},\,x_{5}\mid[x_{1},x_{2}]=x_{3},\,[x_{1},x_{3}]=x_{4},\,[x_{1},x_{4}]=[x_{2},x_{3}]=x_{5}\rangle.
	\]
	
	Then
	\[\zeta_{\text{Fil}_{4}(\Fp)}^{\triangleleft}(s)=\zeta_{M_{4}(\Fp)}^{\triangleleft}(s)=1+(1+p)t+t^{2}+t^{3}+t^{4}+t^{5}.\]
\end{cor}

\begin{rem}
	Over $\Zp$, we have $\zeta_{\text{Fil}_{4}(\Zp)}^{\triangleleft}(s)\neq\zeta_{M_{4}(\Zp)}^{\triangleleft}(s)$.
	This demonstrates that the approximation over $\Fp$ is not enough to spot the subtle difference between $\textrm{Fil}_4$ and $M_4$. 
\end{rem}

\subsubsection{The maximal class-$c$ Lie algebra $M_{c}$, counting subalgebras}

\begin{thm} \label{thm:fp.mc}
	\begin{align*}
		\zeta_{M_{2}(\Fp)}^{\leq}(s)=&1+(1+p)t+(1+p+p^{2})t^{2}+t^3,
	\end{align*}
	and in general for $c\geq3$,
	\begin{align*}
		\zeta_{M_{c}(\Fp)}^{\leq}(s)=&\zeta_{M_{c-1}(\Fp)}^{\leq}(s)+\sum_{a=0}^{c-1}\binom{c-1}{a}_{p}p^{c-1-a}t^{a+2}+p^{c}t^{c}.
	\end{align*}
\end{thm}
\begin{proof}
	For $c=2$, again we get 
	\[\zeta_{M_{2}(\Fp)}^{\leq}(s)=\zeta_{H(\Fp)}^{\leq}(s)=1+(1+p)t+(1+p+p^{2})t^{2}+t^3.\]
	
	For $c\geq3$, let
	\[
	M=\left(\begin{array}{cccc}
		m_{1,1} & m_{1,2} & \cdots & m_{1,c+1}\\
		0 & m_{2,2} & \cdots & m_{2,c+1}\\
		0 & 0 & \ddots & \vdots\\
		0 & 0 & \cdots & m_{c+1,c+1}
	\end{array}\right)
	\]
	denote a generator matrix for a subspace of $M_{c}(\Fp)$ as before. Note that for subalgebra, since $[e_{i},e_{j}]=0$ for all $2\leq i,j\leq c+1$, we only need to consider
	\[[\boldsymbol{m}_{1},\boldsymbol{m}_{i}]=m_{1,1}(m_{i,i}e_{i+1}+m_{i,i+1}e_{i+2}+\cdots+m_{i,c}e_{c+1}),\]
	for $2\leq i \leq c$. Hence the diagonal conditions are $m_{i+1,i+1}|m_{1,1}m_{i,i}$ for all $2\leq i \leq c$. 
	
	Let $c=3$.  The diagonal conditions for $M\leq M_{3}$ are $m_{4,4}|m_{1,1}m_{3,3}$ and $m_{3,3}|m_{1,1}m_{2,2}$. 
	If $m_{4,4}=1$, then 
	\[
	M=\left(\begin{array}{cccc}
		m_{1,1} & m_{1,2} & m_{1,3}&0 \\
		0 & m_{2,2} & m_{2,3}&0 \\
		0 & 0 & m_{3,3} &0\\
		0&0&0&1
	\end{array}\right)
	\]
	is a subalgebra of  $M_{3}(\Fp)$ if and only if 
	\[
	M'=\left(\begin{array}{ccc}
		m_{1,1} & m_{1,2} & m_{1,3} \\
		0 & m_{2,2} & m_{2,3} \\
		0 & 0 & m_{3,3} 
	\end{array}\right)
	\]
	is the subalgebra of $M_{2}(\Fp)$. This gives $\zeta_{M_{2}(\Fp)}^{\leq}(s)$.
	
	If $m_{4,4}=0$, since we have $m_{4,4}|m_{1,1}m_{3,3}$ at least one of $m_{1,1}$ or $m_{3,3}$ also needs to be zero. Hence we do not get any subalgebras of index $p$.
	
	For subalgebras of index $p^{2}$, note that by our diagonal conditions the only possible case is 
	\[
	M=\left(\begin{array}{cccc}
		0 & 0 & 0&0 \\
		0 & 1 & 0&m_{2,4} \\
		0 & 0 & 1 &m_{3,4}\\
		0&0&0&0
	\end{array}\right).
	\]
	This gives $p^{2}$ subalgebras of index $p^{2}$.
	
	For subalgebras of index $p^{3}$, we have 
	\[\left(\begin{array}{cccc}
		0 & 0 & 0&0 \\
		0 & 0 & 0&0 \\
		0 & 0 & 1 &m_{3,4}\\
		0&0&0&0
	\end{array}\right), \left(\begin{array}{cccc}
		0 & 0 & 0&0 \\
		0 & 1 & m_{2,3}&m_{2,4} \\
		0 & 0 & 0 &0\\
		0&0&0&0
	\end{array}\right),\,\textrm{and}\, \left(\begin{array}{cccc}
		1 & m_{1,2} & m_{1,3} &m_{1,4} \\
		0 & 0 & 0&0 \\
		0 & 0 & 0 &0\\
		0&0&0&0
	\end{array}\right),\]
	and they give $p+p^{2}+p^{3}$ subalgebras of index $p^{3}$. Finally we get one trivial  subalgebra of index $p^{4}$. 
	Hence we get 
	\begin{align*}
		\zeta_{M_{3}(\Fp)}^{\leq}(s)&=\zeta_{M_{2}(\Fp)}^{\leq}(s)+p^{2}t^{2}+(p+p^{2}+p^{3})t^{3}+t^{4}\\
		&=\zeta_{M_{2}(\Fp)}^{\leq}(s)+\sum_{a=0}^{2}\binom{2}{a}_{p}p^{2-a}t^{a+2}+p^{3}t^{3}
	\end{align*}
as required. Now fix $c\geq4$. Write
	\[M=\left(\begin{array}{ccccc}
		m_{1,1} & m_{1,2} & \cdots&m_{1,c}&m_{1,c+1} \\
		0 & m_{2,2} &\cdots&m_{2,c}&m_{2,c+1} \\
		\vdots & \vdots & \ddots &\vdots&\vdots\\
		0&0&\cdots&m{c,c}&m_{c,c+1}\\
		0&0&\cdots&0&m_{c+1,c+1}
	\end{array}\right)=\left(\begin{array}{c|c}
		M'&\begin{matrix}m_{1,c+1}\\\vdots\\m_{c,c+1}\end{matrix}\\
		\hline 0&m_{c+1,c+1}\end{array}\right) \]
	
	where
	\[
	M'=\left(\begin{array}{cccc}
		m_{1,1} & m_{1,2} & \cdots&m_{1,c} \\
		0 & m_{2,2} &\cdots&m_{2,c} \\
		\vdots & \vdots & \ddots &\vdots\\
		0&0&\vdots&m_{c,c}
	\end{array}\right).
	\]
	For $m_{c+1,c+1}=1$, 
	\[M=\left(\begin{array}{c|c}
		M'&\begin{matrix}0\\\vdots\\0\end{matrix}\\
		\hline 0&1\end{array}\right)\]gives $\zeta_{M_{c-1}(\Fp)}^{\leq}(s).$
	
	For $m_{c+1,c+1}=0$, again our diagonal conditions $m_{i+1,i+1}|m_{1,1}m_{i,i}$ allow no subalgebras of index $p$. Suppose $m_{1,1}=1$. Then we need to have $m_{c+1,c+1}=m_{c,c}=\ldots=m_{2,2}=0$, which means we need at least $c$ zeroes. Thus $m_{1,1}$ has to be equal to 0 for subalgebras of index $p^{i}$ where $2\leq i\leq c-1$ or $i=c+1$. 
	
	For $m_{1,1}=0$, since
	\[[\bsm_{i},\bsm_{j}]=0\]
	for all $1\leq i<j\leq n$, we can put whatever we want in $(c-1)\times(c-1)$ upper triangular matrices
	\[
	M''=\left(\begin{array}{cccc}
		m_{2,2} & m_{2,3} & \cdots&m_{2,c} \\
		0 & m_{3,3} &\cdots&m_{3,c} \\
		\vdots & \vdots & \ddots &\vdots\\
		0&0&\vdots&m_{c,c}
	\end{array}\right).
	\]
	Such $M$ would look like
	\begin{align*}
		M&=\left(\begin{array}{ccccc}
			0 & 0& \cdots&0&0 \\
			0 & m_{2,2} &\cdots&m_{2,c}&m_{2,c+1} \\
			\vdots & \vdots & \ddots &\vdots&\vdots\\
			0&0&\cdots&m_{c,c}&m_{c,c+1}\\
			0&0&\cdots&0&0
		\end{array}\right).
	\end{align*}
	Let $a$ denote the total number of zeroes in the diagonal entries of $M''$. We know there are $\binom{c-1}{a}_{p}$ such $M''$, and each of them gives additional $p^{c-1-a}$ possibilities of $M$. This gives
	\[\binom{c-1}{a}_{p}\cdot p^{c-1-a}\]
	subalgebras of index $p^{a+2}$ for $0\leq a\leq c-1$. 
	Finally we get $p^{c}$ subalgebras of index $p^{c}$ from
	\begin{align*}
		M&=\left(\begin{array}{ccccc}
			1 & m_{1,2}& \cdots&m_{1,c}&m_{1,c+1} \\
			0 & 0 &\cdots&0&0 \\
			\vdots & \vdots & \ddots &\vdots&\vdots\\
			0&0&\cdots&0&0\\
			0&0&\cdots&0&0
		\end{array}\right).
	\end{align*}
	Thus we get
	\begin{align*}
		\zeta_{M_{c}(\Fp)}^{\leq}(s)=&\zeta_{M_{c-1}(\Fp)}^{\leq}(s)+\sum_{a=0}^{c-1}\binom{c-1}{a}_{p}p^{c-1-a}t^{a+2}+p^{c}t^{c}
	\end{align*}
	as required.
\end{proof}

\begin{rem} For primes $p>c$, our formula matches with a result by Berkovich \cite{Berkovich/90} that if $G$ is a group of order $p^{m}$ and exponent $p$, then the number of subgroups of order $p^{n}$ for $2\leq n\leq m-2$ is congruent to $1+p+2p^{2}$ modulo $p^{3}$. Similar to the ideal case, the local subalgebra zeta function $\zeta_{M_{c}(\Zp)}^{\leq}(s)$	is only known for $c\leq4$, first computed in \cite{Taylor/01}.
\end{rem}

\subsubsection{$\mff_{c,d}$, the free nilpotent Lie ring}
Let $\mff_{c,d}$ be the free class-$c$ nilpotent Lie ring on $d$ generators. 

\begin{thm}[d=2]\label{thm:fc2}
	\begin{align*}
		\zeta_{\mff_{2,2}(\Fp)}^{\triangleleft}(s)&=1+(1+p)t+t^{2}+t^{3},\\
		\zeta_{\mff_{3,2}(\Fp)}^{\triangleleft}(s)&=\zeta_{\mff_{2,2}(\Fp)}^{\triangleleft}(s)+(1+p)t^{4}+t^{5},\\
		\zeta_{\mff_{4,2}(\Fp)}^{\triangleleft}(s)&=\zeta_{\mff_{3,2}(\Fp)}^{\triangleleft}(s)+(p+p^{2})t^{5}+(1+p+p^{2})t^{6}+(1+p+p^{2})t^{7}+t^{8}.
	\end{align*}
\end{thm}
\begin{proof}
	First, recall that $\mff_{2,2}(\Fp)=H(\Fp)$. Hence we have \[\zeta_{\mff_{2,2}(\Fp)}^{\triangleleft}(s)=\zeta_{H(\Fp)}^{\triangleleft}(s)=1+(1+p)t+t^{2}+t^{3}.\]
For $\mff_{3,2}(\Fp)$, note that the dimension of $\mff_{3,2}(\Fp)$ is 5, with $d_{1}=2, d_{2}=1, d_{3}=2$.	
	Let	
	\begin{align*}
		M&=\left(\begin{array}{cc|c|cc}
			m_{1,1} & m_{1,2}& m_{1,3}&m_{1,4}&m_{1,5} \\
			0 & m_{2,2} &m_{2,3}&m_{2,4}&m_{2,5} \\
			\hline 0 & 0 & m_{3,3} &m_{3,4}&m_{3,5}\\
			\hline 0&0&0&m_{4,4}&m_{4,5}\\
			0&0&0&0&m_{5,5}
		\end{array}\right).
	\end{align*}
	Suppose $m_{4,4}=m_{5,5}=1$. Then this corresponds to 
	\begin{align*}
		M&=\left(\begin{array}{cc|c|cc}
			m_{1,1} & m_{1,2}& m_{1,3}&0&0 \\
			0 & m_{2,2} &m_{2,3}&0&0 \\
			\hline 0 & 0 & m_{3,3} &0&0\\
			\hline 0&0&0&1&0\\
			0&0&0&0&1
		\end{array}\right),
	\end{align*}
	which gives $\zeta_{\mff_{2,2}(\Fp)}^{\triangleleft}(s)$.
	
	Since $m_{4,4}\mid m_{1,1},m_{3,3}$ and $m_{5,5}|m_{2,2},m_{3,3}$, the rest  gives $(1+p)t^{4}+t^{5}$. Hence we get
	\[\zeta_{\mff_{3,2}(\Fp)}^{\triangleleft}(s)=\zeta_{\mff_{2,2}(\Fp)}^{\triangleleft}(s)+(1+p)t^{4}+t^{5}.\]
	
	Similarly, for $\mff_{4,2}$ we have 
	\begin{align*}
		M&=\left(\begin{array}{cc|c|cc|ccc}
			m_{1,1} & m_{1,2}& m_{1,3}&m_{1,4}&m_{1,5}&m_{1,6}&m_{1,7}&m_{1,8} \\
			0 & m_{2,2} &m_{2,3}&m_{2,4}&m_{2,5}&m_{2,6}&m_{2,7}&m_{2,8} \\
			\hline 0 & 0 & m_{3,3} &m_{3,4}&m_{3,5}&m_{3,6}&m_{3,7}&m_{3,8}\\
			\hline 0&0&0&m_{4,4}&m_{4,5}&m_{4,6}&m_{4,7}&m_{4,8}\\
			0&0&0&0&m_{5,5}&m_{5,6}&m_{5,7}&m_{5,8}\\
			\hline 0&0&0&0&0&m_{6,6}&m_{6,7}&m_{6,8}\\
			0&0&0&0&0&0&m_{7,7}&m_{7,8}\\
			0&0&0&0&0&0&0&m_{8,8}
		\end{array}\right)
	\end{align*}
	and  we get the diagonal conditions for $M\triangleleft \mff_{4,2}(\Fp)$:
	\begin{align*}
		m_{6,6}|&m_{1,1},m_{2,2},m_{3,3},m_{4,4},\\
		m_{7,7}|&m_{1,1},m_{2,2},m_{3,3},m_{4,4},m_{5,5},\\
		m_{8,8}|&m_{1,1},m_{2,2},m_{3,3},m_{5,5}.
	\end{align*}
	Therefore one can check 
	\begin{align*}
		&(m_{6,6},m_{7,7},m_{8,8})=(1,1,1) \textrm{ gives }  \zeta_{F_{3,2}(\Fp)}^{\triangleleft}(s),\\
		&(m_{6,6},m_{7,7},m_{8,8})=(0,1,1) \textrm{ gives } pt^{5}+t^{6},\\
		&(m_{6,6},m_{7,7},m_{8,8})=(1,0,1) \textrm{ gives } pt^{6},\\
		&(m_{6,6},m_{7,7},m_{8,8})=(1,1,0) \textrm{ gives } p^{2}t^{5}+p^{2}t^{6},\\
		&(m_{6,6},m_{7,7},m_{8,8})=(0,0,1) \textrm{ gives } t^{7},\\
		&(m_{6,6},m_{7,7},m_{8,8})=(0,1,0) \textrm{ gives } pt^{7},\\
		&(m_{6,6},m_{7,7},m_{8,8})=(1,0,0)\textrm{ gives } p^{2}t^{7},\\
		&(m_{6,6},m_{7,7},m_{8,8})=(0,0,0) \textrm{ gives } t^{8},
	\end{align*}
giving 
\[\zeta_{\mff_{4,2}(\Fp)}^{\triangleleft}(s)=\zeta_{\mff_{3,2}(\Fp)}^{\triangleleft}(s)+(p+p^{2})t^{5}+(1+p+p^{2})t^{6}+(1+p+p^{2})t^{7}+t^{8}.\]
\end{proof}
\begin{rem}
	Similar to previous examples, we do not know $\zeta_{\mff_{4,2}(\Zp)}^{\triangleleft}(s)$ yet. We do not even know whether $\zeta_{\mff_{4,2}}^{\triangleleft}(s)$ is ($\Zp$-)uniform, finitely uniform, or non-uniform. Our computation of $\zeta_{\mff_{4,2}(\Fp)}^{\triangleleft}(s)$, which shows that $\zeta_{\mff_{4,2}}^{\triangleleft}(s)$ is $\Fp$-uniform,  provides the first step to compute and understand $\zeta_{\mff_{4,2}(\Zp)}^{\triangleleft}(s)$, which plays an important role regarding Higman's PORC conjecture.
\end{rem}

\subsection{Additional approach for class-2  nilpotent Lie algebras}\label{subsec:c2}
In \cite{Voll/04,Voll/05,Voll/11}, Voll showed how one can compute ideal zeta functions of class-2 Lie rings based on an enumeration of vertices in the affine Bruhat-Tits building associated to $\SL_{n}(\Qp)$. We adapt this method for class-2 $\Fp$-Lie algebras. 

\begin{lem}\label{lem:idealc2}
	Let $L$ be a class-2 Lie ring with $L'=[L,L],\,L/L'\cong L_{1}$ and $L'\cong L_{2}$. Then we have 
	\begin{align*}\zeta_{\LFp}^{\triangleleft}(s)=&\zeta_{\mathbb{F}^{d_{1}}_{p}}(s)+p^{-ds}\\
		&+\sum_{0<\Lambda_{2}< L_{2}}\left|L_{2}:\Lambda_{2}\right|^{-s}\left|L_{1}:X(\Lambda_{2})\right|^{-s}\sum_{\Lambda_{1}\leq X(\Lambda_{2})}\left|L_{2}:\Lambda_{2}\right|^{\textrm{rk}(\Lambda_{1})}\left|X(\Lambda_{2}):\Lambda_{1}\right|^{-s},
	\end{align*}
	where $X(\Lambda_{2})/\Lambda_{2}=Z(L/\Lambda_{2})$ and dim$(L)=d$, dim$(L_{1})=d_{1}$ and dim$(L_{2})=d_{2}$.
\end{lem}

\begin{proof}
	Let $\Lambda$ be a finite dimensional subspace in  $\LFp$. Then, like over $\mathbb{Z}_p$, it gives rise to a pair $(\Lambda_{1},\,\Lambda_2)$ of a subspace 
	\[\Lambda_{2}:=\Lambda\cap L_{2}(\Fp)\leq  L_{2}(\Fp)\]
	in the derived algebra and a subspace 
	\[\Lambda_{1}:=(\Lambda+  L_{2}(\Fp))/ L_{2}(\Fp)\leq \LFp/  L_{2}(\Fp).\]
	
	The crucial difference between $\mathbb{Z}_{p}$ and $\mathbb{F}_{p}$ is that over $\mathbb{F}_{p}$, this surjective map is not $\left| L_{2}(\Fp):\Lambda_{2}\right|^{d_1}$ to 1  but $\left| L_{2}(\Fp):\Lambda_{2}\right|^{\dim(\Lambda_{1})}$ to 1. We get $\Lambda\triangleleft \LFp$ if and only if its associated pair $(\Lambda_{1},\,\Lambda_{2})$ is admissible. More precisely, let 
	\[X(\Lambda_{2})/\Lambda_{2}:=Z(\LFp/\Lambda_{2}).\] Then 
	$(\Lambda_{1},\,\Lambda_{2})$ is admissible if and only if $\Lambda_{1}\leq X(\Lambda_{2})$. Therefore, we have 
	\begin{align*}
		\zeta_{\LFp}^{\triangleleft}(s)=&\sum_{\Lambda_{2}\leq  L_{2}(\Fp)}\left| L_{2}(\Fp):\Lambda_{2}\right|^{-s}\left|\LFp:X(\Lambda_{2})\right|^{-s}\\
		&\cdot\sum_{\Lambda_{1}\leq X(\Lambda_{2})}\left| L_{2}(\Fp):\Lambda_{2}\right|^{\dim(\Lambda_{1})}\left|X(\Lambda_{2}):\Lambda_{1}\right|^{-s}.
	\end{align*}
	This is analogous to Lemma 6.1 in \cite{GSS/88}.  Note that since we are over $\mathbb{F}_{p}$, each $\Lambda_{2}$ is already maximal. Hence we first assume that $\Lambda_{2}$ is of elementary divisor type 
	\[(\underbrace{0,\ldots,0}_{i},\underbrace{1,\ldots,1}_{d_{2}-i}),\] where $i$ is the codimension of $\Lambda_{2}$. Let
	\begin{align*}A_{i}^\triangleleft(s):=&\sum_{\dim(\Lambda_{2})=d_{2}-i}\left| L_{2}(\Fp):\Lambda_{2}\right|^{-s}\left|\LFp:X(\Lambda_{2})\right|^{-s}\\
		&\cdot\sum_{\Lambda_{1}\leq X(\Lambda_{2})}\left| L_{2}(\Fp):\Lambda_{2}\right|^{\dim(\Lambda_{1})}\left|X(\Lambda_{2}):\Lambda_{1}\right|^{-s}.
	\end{align*}
	Then one can see that 
	
	\[A_{i}^\triangleleft(s)=\sum_{\dim(\Lambda_{2})=d_{2}-i}p^{-is}\left|\LFp:X(\Lambda_{2})\right|^{-s}\sum_{j=0}^{\textrm{rk}(X(\Lambda_{2}))}\binom{\textrm{rk}(X(\Lambda_{2}))}{j}_{p}\,p^{ij-(\rk(X(\Lambda_{2}))-j)s}\]
	and
	\[\zeta_{\LFp}^{\triangleleft}=\zeta_{\mathbb{F}^{d_{1}}_{p}}+p^{-ds}+\sum_{i=1}^{d_{2}-1}A_{i}^{\triangleleft}(s).\]
\end{proof}

Since $\dim(X(\Lambda_{2}))=d_{1}-\log_{p}\left(\left|\LFp:X(\Lambda_{2})\right|\right)$, we only need to understand the map $\Lambda_{2}\mapsto\left|\LFp:X(\Lambda_{2})\right|$. 

\subsubsection{The Grenham Rings $\mathcal{G}_{n}$}

Let $\mathcal{G}_{n}=\langle w,x_{1},\ldots,x_{n-1},y_{1},\ldots,y_{n-1}\mid\;[w,x_{i}]=y_{i}\rangle$, denote the Grenham Lie ring on $n$-generators. We have

\begin{thm}\label{thm:fp.grenhamideal}
	\[\begin{aligned}
		\zeta_{\mathcal{G}_{n}(\Fp)}^{\triangleleft}(s)=&\zeta_{\mathbb{F}_{p}^{n}}(s)+p^{-(2n-1)s}\\
		&+\sum_{i=1}^{n-2}\binom{n-1}{i}_{p}\, p^{-(2i+1)s}\sum_{k=0}^{n-(i+1)}\binom{n-(i+1)}{k}_{p}\,p^{ik-(n-(i+1)-k)s}.
	\end{aligned}\]
\end{thm}

\begin{proof}
	Recall 
	\[\zeta_{\LFp}^{\triangleleft}=\zeta_{\mathbb{F}^{d_{1}}_{p}}+p^{-ds}+\sum_{i=1}^{d_{2}-1}A_{i}^{\triangleleft}(s).\]
	For $\mathcal{G}_{n}$ we have $d_1=n,d_2=n-1$ and $d=2n-1$. So we have   
	\[\zeta_{\mathcal{G}_{n}(\Fp)}^{\triangleleft}=\zeta_{\mathbb{F}_{p}^{n}}(s)+p^{-(2n-1)s}+\sum_{i=1}^{n-2}A_{i}^{\triangleleft}(s).\]
	
	Now, one needs to compute \[A_{i}^\triangleleft(s)=\sum_{\dim(\Lambda_{2})=n-1-i}p^{-is}\left|\LFp:X(\Lambda_{2})\right|^{-s}\sum_{j=0}^{\dim(X(\Lambda_{2}))}\binom{\dim(X(\Lambda_{2}))}{j}_{p}\,p^{ij-(\dim(X(\Lambda_{2}))-j)s}.\]
	In \cite[Chapter 5]{Voll/02}, Voll already showed that for $\mathcal{G}_{n}$, we have $\log_{p}(\left|\mathcal{G}_{n}(\Fp):X(\Lambda_{2})\right|)=i+1$ and $\dim(X(\Lambda_2))=n-(i+1)$ when $\dim(\Lambda_{2})=n-(i+1)$, as they only depend on the elementary divisor type of the maximal element. Hence we get the required results.
\end{proof}

\section{Ideal zeta functions of $\Fp$-Lie algebras and Higman's PORC conjecture}\label{sec:4}

As mentioned earlier, one of the interesting applications of the zeta functions of nilpotent Lie rings is their connection to Higman's PORC conjecture. Here we compute ideal zeta functions of three $\Fp$-Lie algebras that were directly involved with the PORC conjecture in previous literature, and discuss what can be said about these computations.

\subsection{du Sautoy's elliptic curve group $L_{E}$ in \cite{duS-ecI/01}}

Let $L_{E}$ be a class-2 Lie ring of rank 9 given by the following
presentation:
\[
L_{E}=\left\langle \begin{aligned}x_{1},\,x_{2},\,x_{3},\,x_{4},\,x_{5},x_{6},\,y_{1},\,y_{2},\,y_{3}\mid[x_{1},x_{5}]=[x_{2},x_{4}]=[x_{3},x_{6}]=y_{1},\\{}
	[x_{1},x_{6}]=[x_{3},x_{4}]=y_{2},\,[x_{1},x_{4}]=[x_{2},x_{5}]=y_{3}
\end{aligned}
\right\rangle .
\]

\begin{thm}\label{thm:Elliptic}
	\[
	\begin{aligned}\zeta_{L_{E}(\Fp)}^{\triangleleft}(s) =& \zeta_{\mathbb{F}_{p}^{6}}(s)+p^{2}\left|E(\mathbb{F}_{p})\right|t^{5}+\left|E(\mathbb{F}_{p})\right|\left(p+p^{2}\right)t^{6}\\
		&+\binom{3}{1}_{p}t^{7} +\binom{3}{2}_{p}t^{8}+t^{9},
	\end{aligned}
	\]
	where $\left|E(\mathbb{F}_{p})\right|=\left|\left\{ (x:y:z)\in\mathbb{P}^{2}(\mathbb{F}_{p})\,:\,y^{2}z=x^{3}-xz^{2}\right\}\right| .$
\end{thm}
\begin{proof} This is essentially the $\Fp$-version of \cite[Theorem 3.2]{duS-ecI/01}.
	Let us first show where $\left|E(\mathbb{F}_{p})\right|$ in the coefficient of $t^{5}$ comes from. 
	
	Let 
	\[
	M=\left(\begin{matrix}1&m_{1,2}&m_{1,3}&0&m_{1,5}&m_{1,6}&0&0&m_{1,9}\\&0&0&0&0&0&0&0&0\\&&0&0&0&0&0&0&0\\&&&1&m_{4,5}&m_{4,6}&0&0&m_{4,9}\\&&&&0&0&0&0&0\\&&&&&0&0&0&0\\&&&&&&1&0&m_{7,9}\\&&&&&&&1&m_{8,9}\\&&&&&&&&0\end{matrix}\right).
	\]
	Then for $1\leq i,j\leq9$ and $k\in{2,3,5,6,9}$, our system of equations $g_{ijk}^{\triangleleft}(m_{r,s})=0$ becomes
	\begin{align*}
		m_{1,5}m_{7,9}+m_{1,6}m_{8,9}&=0,&m_{1,5}&=0,&m_{1,6}m_{7,9}&=0,\\
		1-m_{1,2}m_{7,9}-m_{1,3}m_{8,9}&=0,&m_{1,2}-m_{7,9}&=0,&-m_{1,3}m_{7,9}-m_{8,9}&=0,\\
		1+m_{4,5}m_{7,9}+m_{4,6}m_{8,9}&=0,&-m_{4,5}+m_{7,9}&=0,&m_{4,6}m_{7,9}+m_{8,9}&=0.\\
	\end{align*}	
	Note that $m_{7,9}\neq0$ since $m_{7,9}=0$ implies $m_{8,9}=0$ and $1+m_{4,5}m_{7,9}+m_{4,6}m_{8,9}=1=0$. Using $m_{7,9}\neq0$, we have	
	\begin{align*}
		m_{1,2}=m_{4,5}&=m_{7,9},&m_{1,5}=m_{1,6}&=0,\\
		m_{1,3}=m_{4,6}&=-\frac{m_{8,9}}{m_{7,9}},&1-m_{7,9}^{2}+\frac{m_{8,9}^{2}}{m_{7,9}}&=0.
	\end{align*}
	
	Hence, whenever $(m_{7,9},m_{8,9})$ is a non-zero point on the elliptic curve $E(\mathbb{F}_{p})$, there are $p^{2}$ matrices determined by $(m_{7,9},m_{8,9})$, where $p^2$ comes from the free choices of $m_{1,9},m_{4,9}\in\Fp$. Thus 
	\[c_{\underline{a},\underline{b}}=\left|E(\mathbb{F}_{p})\right|p^{2},\]
	and we get $\left|E(\mathbb{F}_{p})\right|$ as required. 
	
	Note that this is a class-2 $\Fp$-Lie algebra, and the dimension of the derived part $d_{2}=3$. So we only need to consider eight cases and we have already done the most difficult one. One can explicitly calculate Theorem \ref{thm:Elliptic} by computing the remaining seven cases.
\end{proof}
\begin{cor}
	$\zeta_{L_{E}}^{\triangleleft}(s)$ is $\Fp$-non-uniform.
\end{cor}
\begin{rem}
	Note that this  $L_{E}(\Fp)$ is the exact same Lie algebra used by du Sautoy and Vaughan-Lee in \cite{DuSVL/2012}. They proved that the number of its \textit{immediate descendants} of order $p^{10}$ (i.e. extension by $\Fp$ of class-3) is not PORC. 
	It is surprising that the wild behavior of the $\Fp$-points on the elliptic curve can still be observed over $\Fp$.
\end{rem}
\begin{rem}
	In \cite{SV/2019arxiv}, Stanojkovski and Voll showed that the current presentation of $L_{E}$ given by du Sautoy is one of the three Hessian representations of the elliptic curve $E:Y^2=X^3-X$. Comparing the zeta functions of three distinct finite $p$-groups given by these three Hessian representations (given by $\boldsymbol{G}_{1,1}(\Fp)$, $\boldsymbol{G}_{2,1}(\Fp)$, and $\boldsymbol{G}_{3,1}(\Fp)$ in \cite{SV/2019arxiv}) would be an interesting exercise.
\end{rem}

Since Higman's PORC conjecture refers specifically to the dependence of $f_{n}(p)$, the number of isomorphism classes of $p$-groups of order $p^{n}$, on residue classes of $p$, one might think studying the behavior of $\zeta_{L(\Fp)}^{*}(s)$ on residue classes might be more appropriate approach. 

\begin{dfn}
	We say that the zeta function $\zeta_{L(\Fp)}^{*}(s)$ is PORC, if there exists a fixed integer $N$ and finitely many polynomials $W_{i}(X,Y)\in\Q[X,Y]$ for $i=[N-1]$ such that for almost all primes $p$, if $p\equiv i\mod N$ then 
	\[\zeta_{L(\Fp)}^{*}(s)=W_{i}(p,p^{-s}).\] 
	We say $\zeta_{L(\Fp)}^{*}(s)$ is non-PORC if there is no such $N$.
\end{dfn}
Our next example shows how the `PORC-ness' of zeta functions of $\Fp$-Lie algebras can be connected to Higman's PORC conjecture.

\subsection{Lee's non-PORC example in \cite{Lee/2016}}
Let  \[
L_{np8}=\left\langle \begin{aligned}x_{1},\,x_{2},\,x_{3},\,x_{4},\,x_{5},\,y_{1},\,y_{2},\,y_{3}\mid[x_{1},x_{4}]=[x_{2},x_{5}]=y_{1},\\{}
	2[x_{1},x_{5}]=[x_{3},x_{4}]=2y_{2},\,[x_{2},x_{4}]=[x_{3},x_{5}]=y_{3}
\end{aligned}
\right\rangle .
\]
\begin{thm}\label{thm:Leenp8} For $p\geq3$,
	\[
	\begin{aligned}\zeta_{L_{np8}(\Fp)}^{\triangleleft}(s) =& \zeta_{\mathbb{F}_{p}^{5}}(s)+p^{3}\left|N(\mathbb{F}_{p})\right|t^{3}+p^{2}\binom{3}{1}_{p}\left|N(\mathbb{F}_{p})\right|t^{4} \\
		&+\left(p+2p^2+p^3\left(1+\left|N(\mathbb{F}_{p})\right|\right)\right)t^{5}+\left(1+p+p^2\left(1+\left|N(\mathbb{F}_{p})\right|\right)\right)t^{6}\\
		&+\binom{3}{1}_{p}t^{7}+t^{8} 
	\end{aligned}
	\]
	where 	
	\[
	\left|N(\mathbb{F}_{p})\right|=\begin{cases}
		3 & p\equiv\textrm{1 mod 3 and \ensuremath{p=a^{2}+27b^{2},}}\\
		0 & p\equiv\textrm{1 mod 3 and \ensuremath{p\neq a^{2}+27b^{2,}}}\\
		1 & p\equiv\textrm{2 mod 3},
	\end{cases}
	\]
	is the number of roots of $N:x^3-2=0$ over $\Fp$. In particular, $\zeta_{L_{np8}(\Fp)}^{\triangleleft}(s)$ is non-PORC.
\end{thm}
\begin{proof}
	As before, let us show where $\left|N(\mathbb{F}_{p})\right|$ in the coefficient of $t^{3}$ comes from. 
	
	Let 
	\[
	M=\left(\begin{matrix}1&0&m_{1,3}&0&m_{1,5}&0&0&m_{1,8}\\&1&m_{2,3}&0&m_{2,5}&0&0&m_{2,8}\\&&0&0&0&0&0&0\\&&&1&m_{4,5}&0&0&m_{4,8}\\&&&&0&0&0&0\\&&&&&1&0&m_{6,8}\\&&&&&&1&m_{7,8}\\&&&&&&&0\\\end{matrix}\right)
	\]
	Then for $1\leq i,j\leq8$ and $k\in{3,5,8}$, our system of equations $g_{ijk}^{\triangleleft}(m_{r,s})=0$ becomes	
	\begin{align*}
		m_{1,5}&=0,&m_{2,5}&=0,&1-m_{1,3}m_{6,8}&=0,\\
		m_{6,8}+m_{1,3}m_{7,8}&=0,&2m_{2,3}m_{6,8}+m_{7,8}&=0,&m_{4,5}-m_{7,8}&=0,\\
		1-m_{2,3}m_{7,8}&=0,&1-m_{4,5}m_{6,8}&=0,&2m_{6,8}+m_{4,5}m_{7,8}&=0.	
	\end{align*}
	
	Note that $m_{7,8}\neq0$, otherwise we get $1=0$. Using $m_{7,8}\neq0$, we have	either
	\begin{align*}
		m_{7,8}&=-\sqrt[3]{2},&m_{1,5}=m_{2,5}&=0,\\
		m_{4,5}&=m_{7,8},&m_{2,3}=m_{6,8}&=\frac{1}{m_{7,8}},
	\end{align*}
	if $\sqrt[3]{2}\in\Fp$, giving $\left|N(\mathbb{F}_{p})\right|$ solutions, or we have no solution if $\sqrt[3]{2}\notin\Fp$.
	Thus 	$c_{\underline{a},\underline{b}}=p^{3}\left|N(\mathbb{F}_{p})\right|$	and we get $\left|N(\mathbb{F}_{p})\right|$ as required. 
	
	Again, this is a class-2 case, and the dimension of the derived part $d_{2}=3$. One can explicitly calculate Theorem \ref{thm:Leenp8} by computing the remaining seven cases.
\end{proof}
\begin{rem} 
	Note that $\zeta_{L_{np8}(\Fp)}^{\triangleleft}(s)$ is non-PORC but $\zeta_{L_{np8}}^{\triangleleft}(s)$ is still $\Fp$-finitely uniform. In \cite{Lee/2016}, the author proved that $L_{np8}(\Fp)$ has non-PORC number of immediate descendants of order $p^9$. This was also due to the non-PORC behavior of $\left|N(\mathbb{F}_{p})\right|$ for $p>3$. The author computed $\zeta_{L_{np8}(\Zp)}^{\triangleleft}(s)$ in \cite{Lee/20arxiv2} as well and showed that $\zeta_{L_{np8}}^{\triangleleft}(s)$ is $\Zp$-finitely uniform, which also involved $\left|N(\mathbb{F}_{p})\right|$. Like the elliptic curve case, the approximation over $\Fp$ was enough to spot this intrinsic non-PORC behavior. 
\end{rem}

However, the next example suggests that this might not be always true.

\subsection{Vaughan-Lee's group $G_{(p,x,y)}$ in \cite{VL/18}}
In \cite{VL/18}, Vaughan-Lee introduced and investigated a family of 3-generator groups $G_{(p,a,b)}$ indexed by a prime $p>3$ and  $a,b\in\Fp$. By the Lazard correspondence, we can construct a family of isomorphic 7 dimensional class-3 Lie algebras $L_{(p,a,b)}(\Fp)=\langle x_1,x_2,x_3,y_1,y_2,y_3,z\rangle$ where
\begin{enumerate}
	\item$L_{(p,x,y)}(\Fp)^{(1)}=\langle x_1,x_2,x_3\rangle,\;L_{(p,x,y)}(\Fp)^{(2)}=\langle y_1,y_2,y_3\rangle,\;
	L_{(p,x,y)}(\Fp)^{(3)}=\langle z\rangle,$
	
	\item The class-2 quotient $L_{(p,x,y)}(\Fp)/\gamma_{3}(L_{(p,x,y)}(\Fp))=\langle x_1,x_2,x_3,y_1,y_2,y_3\rangle\cong\mff_{2,3}(\Fp)$ with 
	$[x_1,x_2]=y_1,\;[x_1,x_3]=y_2,\;[x_2,x_3]=y_3,$ and
	\item $[[x_1,x_3],x_3]=[[x_1,x_2],x_1]=z$, $[[x_2,x_3],x_2]=bz$, $[[x_2,x_3],x_3]=az$ and all other commutators are trivial.
\end{enumerate}
\begin{thm} \label{thm:vl}
	\begin{enumerate}
		\item When $b\neq0$, we have
		\[
		\begin{aligned}
			\zeta_{L_{(p,a,b)}(\Fp)}^{\triangleleft}(s)=&1+(1+p+p^{2})t+(1+p+p^{2})t^{2}+(1+p+p^{2}+p^{3})t^{3}\\
			&+(1+p+p^{2})t^{4}+(1+p+p^{2})t^{5}+t^{6}+t^{7},
		\end{aligned}\]and
		\item when $b=0$, we have
		\[
		\begin{aligned}
			\zeta_{L_{(p,a,0)}(\Fp)}^{\triangleleft}(s)=&1+(1+p+p^{2})t+(1+p+p^{2})t^{2}+(1+p+p^{2}+p^{3})t^{3}\\
			&+(1+p+p^{2})t^{4}+(1+p+p^{2})t^{5}+(1+p)t^{6}+t^{7}.
		\end{aligned}\]	
	\end{enumerate}
\end{thm}
\begin{proof}
	Let \begin{align*}
		M&=\left(\begin{array}{ccc|ccc|c}
			m_{1,1} & m_{1,2}& m_{1,3}&m_{1,4}&m_{1,5}&m_{1,6}&m_{1,7} \\
			0 & m_{2,2} &m_{2,3}&m_{2,4}&m_{2,5}&m_{2,6}&m_{2,7} \\
			0 & 0 & m_{3,3} &m_{3,4}&m_{3,5}&m_{3,6}&m_{3,7}\\
			\hline 0&0&0&m_{4,4}&m_{4,5}&m_{4,6}&m_{4,7}\\
			0&0&0&0&m_{5,5}&m_{5,6}&m_{5,7}\\
			0&0&0&0&0&m_{6,6}&m_{6,7}\\
			\hline0&0&0&0&0&0&m_{7,7}
		\end{array}\right).
	\end{align*}
	We consider 2 possible values of $m_{7,7}$. Suppose $m_{7,7}=1$. Then since the class-2 quotient of $L_{(p,a,b)}(\Fp)$ is $\mff_{2,3}(\Fp)$, we get (cf. Theorem \ref{thm:fp.f2d})
	\[
	\begin{aligned}
		\zeta_{\mff_{2,3}(\Fp)}^{\triangleleft}(s)=&1+(1+p+p^{2})t+(1+p+p^{2})t^{2}+(1+p+p^{2}+p^{3})t^{3}\\
		&+(1+p+p^{2})t^{4}+(1+p+p^{2})t^{5}+t^{6}.
	\end{aligned}\]

	Now suppose $m_{7,7}=0$. We have 
	\begin{align*}
		[\bsm_{4},e_{1}]&=m_{4,4}e_{7},&[\bsm_{4},e_{2}]&=bm_{4,6}e_{7},&[\bsm_{5},e_{2}]&=bm_{5,6}e_{7},\\
		[\bsm_{6},e_{2}]&=bm_{6,6}e_{7},&[\bsm_{4},e_{3}]&=(m_{4,5}+am_{4,6})e_{7},&[\bsm_{5},e_{3}]&=(m_{5,5}+am_{5,6})e_{7},\\
		[\bsm_{6},e_{3}]&=am_{6,6}e_{7}.
	\end{align*}
	When $b\neq0$, these conditions force $m_{1,1}=m_{2,2}=\ldots=m_{6,6}=0$, giving one ideal of index $p^{7}$. This proves (1) of our theorem.
	
	When $b=0$ and $a\neq0$, these conditions force  $m_{1,1}=m_{2,2}=m_{3,3}=m_{4,4}=m_{6,6}=0$, but leave $m_{5,5}$ free:
	
	\begin{align*}
		M&=\left(\begin{array}{ccc|ccc|c}
			0 & 0& 0&0&0&0&0 \\
			0 & 0& 0&0&0&0&0 \\
			0 & 0& 0&0&0&0&0 \\
			\hline 	0 & 0& 0&0&0&0&0 \\
			0&0&0&0&m_{5,5}&m_{5,6}&m_{5,7}\\
			0 & 0& 0&0&0&0&0 \\
			\hline0&0&0&0&0&0&0
		\end{array}\right).
	\end{align*}
	If $m_{5,5}=0$, we have one ideal of index $p^{7}$ as before. If $m_{5,5}=1$, our formula gives
	\begin{align*}
		M&=\left(\begin{array}{ccc|ccc|c}
			0 & 0& 0&0&0&0&0 \\
			0 & 0& 0&0&0&0&0 \\
			0 & 0& 0&0&0&0&0 \\
			\hline 	0 & 0& 0&0&0&0&0 \\
			0&0&0&0&1&-\frac{1}{a}&m_{5,7}\\
			0 & 0& 0&0&0&0&0 \\
			\hline0&0&0&0&0&0&0
		\end{array}\right),
	\end{align*} 
	giving another $p$ ideals of index $p^{6}$ ($-\frac{1}{a}\in\Fp$ is well-defined and unique since $a\neq0$).
	
	When $a=b=0$, then these conditions force  $m_{1,1}=m_{2,2}=m_{3,3}=m_{4,4}=m_{5,5}=0$, but leave $m_{6,6}$ free:
	
	\begin{align*}
		M&=\left(\begin{array}{ccc|ccc|c}
			0 & 0& 0&0&0&0&0 \\
			0 & 0& 0&0&0&0&0 \\
			0 & 0& 0&0&0&0&0 \\
			\hline 	0 & 0& 0&0&0&0&0 \\
			0&0&0&0&0&0&0\\
			0 & 0& 0&0&0&m_{6,6}&m_{6,7} \\
			\hline0&0&0&0&0&0&0
		\end{array}\right).
	\end{align*}
	One can easily see that as before, $m_{6,6}=1$ gives $p$ ideals of index $p^6$, and $m_{6,6}=0$ gives one ideal of index $p^7$, which proves part (2).
\end{proof}

\begin{rem}\label{rem:vl.galois}
	Vaughan-Lee proved in \cite[Theorem 2]{VL/18} that if we choose integers $a,b\neq0$ coprime to $p$ such that the Galois group of the polynomial $g(T)=T^3-aT-b$ over the rationals is $S_{3}$, then the order of the automorphism group of $G_{(p,a,b)}$ is non-PORC. Our result shows that enumerating the ideals of  $L_{(p,a,b)}(\Fp)$ (or the normal subgroups of  $G_{(p,a,b)}$) does not capture this subtle non-PORC behavior of the automorphisms. Unfortunately, we do not know $\zeta_{L_(p,a,b)(\Zp)}^{\triangleleft}(s)$ to compare yet.
\end{rem}

\section{More explicit computations}\label{sec:5}
In this section, we record (without explicit proofs) more examples of various zeta functions obtained by the method developed in Section \ref{sec:2} and Section \ref{sec:3}.

\subsection{Nilpotent $\Fp$-Lie algebras}
\begin{thm}[Free class-2 nilpotent Lie rings on $d$-generators]\label{thm:fp.f2d}
	Let $\mff_{2,d}$ denote the free class-2 nilpotent Lie ring on $d$ generators. For $k\in\mathbb{N}$, let \[I_{k}:=\left\{i\in\mathbb{N}:\binom{k-1}{2}<i\leq\binom{k}{2}\right\},\] and let $\pi:\mathbb{N}\rightarrow\mathbb{N},$ where $\pi(i)=k$ for $i\in I_{k}.$ We have
	\[\begin{aligned}
		\zeta_{\mff_{2,d}(\Fp)}^{\triangleleft}(s)=&\zeta_{\mathbb{F}_{p}^{d}}(s)+p^{-(d+d')s}\\
		&+\sum_{i=1}^{d'-1}\binom{d'}{i}_{p}\, p^{-(i+\pi(i))s}\sum_{k=0}^{d-\pi(i)}\binom{d-\pi{i}}{k}_{p}\,p^{ik-(d-\pi(i)-k)s}.
	\end{aligned}\]
\end{thm}

\begin{exm}	\[
	\begin{aligned}
		\zeta_{\mff_{2,3}(\Fp)}^{\triangleleft}(s)=&\zeta_{\mathbb{F}_{p}^{3}}(s)+t^{6}+(1+p+p^{2})t^{3}(p+t)+(1+p+p^{2})t^{5}\\	
		=&1+(1+p+p^{2})t+(1+p+p^{2})t^{2}+(1+p+p^{2}+p^{3})t^{3}\\
		&+(1+p+p^{2})t^{4}+(1+p+p^{2})t^{5}+t^{6}.
	\end{aligned}\]
	
\end{exm}
\begin{thm} Let $H_{m}$ denote the central product of $m$ copies of Heisenberg Lie ring. Then
	\[\zeta_{H_{m}(\Fp)}^{\triangleleft}(s)=\zeta_{\mathbb{F}_{p}^{2m}}(s)+t^{2m+1}.\]
\end{thm}

\begin{thm}
	Let 
	\[
	g_{5,3}=\langle x_{1},\,x_{2},\,x_{3},\,x_{4},\,x_{5}\mid[x_{1},x_{2}]=x_{4},\,[x_{1},x_{4}]=[x_{2},x_{3}]=x_{5}\rangle.
	\]
	Then
	\[
	\begin{aligned}\zeta_{g_{5,3}(\Fp)}^{\triangleleft}(s) & =\zeta_{\mathbb{F}_{p}^{3}}(s)+pt^{3}+t^{4}+t^{5}.
	\end{aligned}
	\]
\end{thm}

\begin{thm}
	Let 
	\[
	g_{6,4}=\langle x_{1},\,x_{2},\,x_{3},\,x_{4},\,y_{1},\,y_{2}\mid[x_{1},x_{2}]=y_{1},\,[x_{1},x_{3}]=[x_{2},x_{4}]=y_{2}\rangle.
	\]
	Then
	\[
	\begin{aligned}\zeta_{g_{6,4}(\Fp)}^{\triangleleft}(s) & =1+(1+p+p^{2}+p^{3})t+(1+p+2p^{2}+p^{3}+p^{4})t^{2}\\
		& +(1+p+2p^{2}+p^{3})t^{3}+(1+p+p^{2})t^{4}\\
		& +(1+p)t^{5}+t^{6}.
	\end{aligned}
	.
	\]
\end{thm}
\subsection{Graded ideal zeta functions of 26 fundamental Lie algebras over $\Fp$ }

In Table \ref{tab:graded} we present $\Fp$-approximations of local graded ideal zeta functions associated with the 26 fundamental graded Lie algebras of dimension at most 6. The first column gives the names of associated $\C$-algebras as in \cite{Kuzmich/99}. One can compare this table with Rossmann's results \cite[Table 2]{Rossmann/16}. 
\begin{table}[htb!]
	\centering
	\begin{tabular}{c|l}
		$L$ & $\zeta_{L(\Fp)}^{\idealgr}(s)$ \\\hline		
		$m3\_2$ & $\zeta_{\Fp^2}(s)+t^3$ \\	
		$m4\_2$ & $\zeta_{\Fp^3}(s)+t^3+t^4$ \\
		$m4\_3$ & $\zeta_{\Fp^2}(s)+t^3+t^4$ \\
		$m5\_2\_1$ & $\zeta_{\Fp^3}(s)+(1+p)t^3+(1+p)t^4+t^5$ \\
		$m5\_2\_2$	 & $\zeta_{\Fp^4}(s)+t^3+(1+p)t^4+t^5$ \\
		$m5\_2\_3$	 & $\zeta_{\Fp^4}(s)+t^5$ \\
		$m5\_3\_1$	 & $\zeta_{\Fp^2}(s)+t^3+(1+p)t^4+t^5$ \\
		$m5\_3\_2$	 & $\zeta_{\Fp^3}(s)+t^3+2t^4+t^5$ \\
		$m5\_4\_1$	& $\zeta_{\Fp^2}(s)+t^3+t^4+t^5$ \\
		$m6\_2\_1$	& $\zeta_{\Fp^3}(s)+(1+p+p^2)t^3+(1+p+p^2)t^4+(1+p+p^2)t^5+t^6$ \\
		$m6\_2\_2$	& $\zeta_{\Fp^4}(s)+(1+p)t^3+(1+2p+p^2)t^4+(2+p)t^5+t^6$ \\
		$m6\_2\_3$	& $\zeta_{\Fp^4}(s)+2t^3+2(1+p)t^4+(1+p)t^5+t^6$ \\
		$m6\_2\_4$	 & $\zeta_{\Fp^4}(s)+t^3+(1+p)t^4+(1+p)t^5+t^6$ \\
		$m6\_2\_5$	& $\zeta_{\Fp^5}(s)+t^3+(1+p+p^2)t^4+(1+p+p^2)t^5+t^6$ \\
		$m6\_2\_6$	 & $\zeta_{\Fp^5}(s)+t^5+t^6$ \\
		$m6\_3\_1$	 & $\zeta_{\Fp^3}(s)+t^3+(2+p)t^4+(2+p)t^5+t^6$ \\
		$m6\_3\_2$	& $\zeta_{\Fp^3}(s)+(1+p)t^(2+p)t^4+2t^5+t^6$ \\
		$m6\_3\_3$	 & same as for $m6\_3\_2$ \\
		$m6\_3\_4$	& $\zeta_{\Fp^3}(s)+(1+p)t^3+(1+p)t^4+t^5+t^6$ \\
		$m6\_3\_5$	 & same as for $m6\_3\_4$ \\
		$m6\_3\_6$	& $\zeta_{\Fp^4}(s)+t^3+(2+p)t^4+(2+p)t^5+t^6$ \\
		$m6\_4\_1$	 & $\zeta_{\Fp^2}(s)+t^3+(1+p)t^4+2t^5+t^6$ \\
		$m6\_4\_2$	& $\zeta_{\Fp^2}(s)+t^3+(1+p)t^4+t^5+t^6$ \\
		$m6\_4\_3$	 & $\zeta_{\Fp^2}(s)+t^3+2t^4+2t^5+t^6$ \\
		$m6\_5\_1$	 & $\zeta_{\Fp^2}(s)+t^3+t^4+t^5+t^6$ \\
		$m6\_5\_2$	& same as for $m6\_5\_1$ \\
	\end{tabular}
	\medskip
	\caption{$\zeta_{L(\Fp)}^{\idealgr}(s)$ for 26 fundamental graded Lie algebras}
	\label{tab:graded}
\end{table}

\subsection{Solvable $\Fp$-Lie algebras $\textrm{tr}_{n}$}\label{subsec:solv}
Here we also record the ideal zeta functions of solvable but non-nilpotent $\Fp$-Lie algebras $\textrm{tr}_{n}$. 

For $n\in\N$, let $\textrm{tr}_{n}$ denote the set of $n\times n$ upper-triangular matrices with the usual Lie bracket $[x,y]=xy-yx$ for $x,y\in\textrm{tr}_{n}$. This Lie ring is solvable, but not nilpotent. In \cite{Woodward/08}, Woodward provided a general formula for $\zeta_{\textrm{tr}_{n}(\Zp)}^{\triangleleft}(s)$, with explicit computations for $n=1,2,3,$ and $4$. In particular he proved that   $\zeta_{\textrm{tr}_{n}}^{\triangleleft}(s)$ is $\Zp$-uniform for all $n\in\N$.

Motivated by these results we record the following computations:

\begin{thm}\label{thm:tr(n)}
	\begin{align*}
		\zeta_{\textrm{tr}_{1}(\Fp)}^{\triangleleft}(s)&=\zeta_{\Fp}(s)=1+t,\\
		\zeta_{\textrm{tr}_{2}(\Fp)}^{\triangleleft}(s)&=\zeta_{\Fp^2}(s)+t^2\zeta_{\Fp}(s)=1+(1+p)t+2t^2+t^3,\\
		\zeta_{\textrm{tr}_{3}(\Fp)}^{\triangleleft}(s)&=\zeta_{\Fp^3}(s)+2t^2\zeta_{\Fp^2}(s)+t^4\zeta_{\Fp}(s)+t^5\zeta_{\Fp}(s)\\
		&=1+(1+p+p^2)t+(3+p+p^2)t^2+(3+2p)t^3+3t^4+2t^5+t^6,\\
		\zeta_{\textrm{tr}_{4}(\Fp)}^{\triangleleft}(s)&=\zeta_{\Fp^4}(s)+3t^2\zeta_{\Fp^3}(s)+3t^4\zeta_{\Fp^2}(s)+t^6\zeta_{\Fp}(s)+2t^5(\zeta_{\Fp^2}(s)+t^2\zeta_{\Fp}(s))\\
		&+t^8\zeta_{\Fp}(s)+t^9\zeta_{\Fp}(s)\\
		&=1+(1+p+p^2+p^3)t+(4+p+2p^2+p^3+p^4)t^2+(4+4p+4p^2+p^3)t^3\\
		&+(7+3p+3p^2)t^4+(8+3p)t^5+(6+2p)t^6+5t^6+3t^8+2t^9+t^{10}.
	\end{align*}
\end{thm}

\section{Further remarks and questions}\label{sec:6}
\subsection{$\Zp$- and $\Fp$-uniformity}
Although we introduced many new explicit computations,  more examples would be useful for further studies. For example, one can ask
\begin{qun}\label{qun:fcd}
	Is $\zeta_{\mff_{c,d}}^{\triangleleft}(s)$ $\Fp$-uniform for all $c$ and $d$?
\end{qun}
As we discussed, for $\Zp$-uniformity this is only known for class-2 cases ($\mff_{2,d}$) and $\mff_{3,2}$. Answering Question \ref{qun:fcd} for more values of $c$ and $d$ will help to understand the uniformity of free Lie rings $\mff_{c,d}$. Also, at the moment we have only very few explicit examples of zeta functions of non-nilpotent Lie rings or algebras. We should develop some techniques, like in Section \ref{subsec:nil.fp}, for non-nilpotent cases. This would help us to compute the zeta functions of $\Fp$-Lie algebras like $\text{sl}_3(\Fp)$ or other solvable or semisimple examples.  

Furthermore, recall that the ideal zeta function $\zeta_{L(\Zp)}^{\ideal}(s)$ of $L$ is equivalent to the zeta functions enumerating the $\Zp$-submodules of $L(\Zp)$ invariant under the \textit{adjoint representations} $\ad(L)$ of $L$. In \cite{LeeVoll2021}, Voll and the author demonstrated how the enumeration of ideals (or graded ideals) in $L(\Zp)$ can be equivalently seen as an enumeration of finite-indexed subrepresentations of integral quiver representations of quivers. We can therefore consider the ideal zeta functions of $\Fp$- of $\Fq$-Lie algebras and their graded approximations from the perspective of \textit{Quiver Grassmannians}.  

In general, one can ask
\begin{qun}
	For any Lie ring $L$, does its $\Zp$-uniformity always equal to its $\Fp$-uniformity?  
\end{qun} 
At the moment we failed to find a single Lie ring $L$ whose $\Zp$-uniformity and $\Fp$-uniformity are different yet. As stated in Remark \ref{rem:Zp.and.Fp}, it would be a major breakthrough  if we can prove that the varieties we need to understand to compute $\zeta_{L(\Fp)}^{*}(s)$ are also among those necessary for $\zeta_{L(\Zp)}^{*}(s)$.

\subsection{Enumerating subalgebras or ideals in some finite quotients of $L(\Zp)$}
For the reasons explained in Remark \ref{rem:fp.lie}, all the results in this article are on $\Fp$-Lie algebras and the corresponding $p$-groups of exponent $p$. However, in terms of the approximations of the zeta functions of $L(\Zp)$,  one can also consider the zeta functions of $L(\Z/p^k\Z)$, enumerating the number of subalgebras or ideals of some finite quotients of $L(\Zp)$, which is clearly different to $L(\mathbb{F}_{p^k})$. Theorem \ref{thm:Elliptic} and \ref{thm:Leenp8} show that for these Lie rings, $L(\Fp)=L(\Z/p\Z)$ was enough to capture the wild behavior of elliptic curve or the non-PORC behavior of cubic residue of 2 inside the ring. However, Corollary \ref{cor:fil4} shows that the first layer  $L(\Fp)$ was not enough to spot somewhat subtler difference between $M_4$ and $\text{Fil}_4$ . In fact, the author showed in \cite[Theorem 6.2.7]{Lee/19thesis} that $\zeta_{\text{Fil}_{4}(\Z/p^k\Z)}^{\triangleleft}(s)\neq\zeta_{M_{4}(\Z/p^k\Z)}^{\triangleleft}(s)$ when $k\geq3$.

To be precise, let $R=\Z/p^{k}\Z$ be a finite quotient of $\Zp$. We say two Lie rings $L$ and $M$ \textit{*-isospectral over $R$} if $L$ and $M$ are non-isomorphic but $\zeta_{L(R)}^{*}(s)=\zeta_{M(R)}^{*}(s)$. For example, in \cite[Theorem 2.55]{duSWoodward/08} du Sautoy and Woodward showed that two non-isomorphic Lie rings, denoted by $\mathfrak{g}_{6,15}$ and $\mathfrak{g}_{6,17}$, satisfy $\zeta_{\mathfrak{g}_{6,15}(\Zp)}^{\ideal}(s)=\zeta_{\mathfrak{g}_{6,17}(\Zp)}^{\ideal}(s)$. So they are $\ideal$-isospectral over $\Zp$. Corollary \ref{cor:fil4} and \cite[Theorem 6.2.7]{Lee/19thesis} tell us that $M_4$ and $\text{Fil}_4$ are  $\ideal$-isospectral over $\Fp=\Z/p\Z$ and $\Z/p^2\Z$, but not over $\Z/p^k\Z$ for $k\geq3$. Since we already know $\zeta_{M_4(\Zp)}^{\ideal}(s)\neq\zeta_{\text{Fil}_4(\Zp)}^{\ideal}(s)$, it is natural to ask how deep we need to dig in before we get a `nice' approximation of $\zeta_{L(\Zp)}^{*}(s)$.
\begin{qun}
	Under what conditions can we predict the local behavior (e.g. the $\Zp$-uniformity, the splitting behavior over $p$, or the *-isospectral over $\Z/p^{k}\Z$) of  $\zeta_{L}^{*}(s)$ just by computing $\zeta_{L(\Fp)}^{*}(s)$ or $\zeta_{L(\Z/p^{k}\Z)}^{*}(s)$ for some $k$? 
\end{qun}

Our current data seems to support the following conjectures:
\begin{con}
	If $L$ is a nilpotent class-2 Lie ring, then $\zeta_{L(\Fp)}^{*}(s)$  and $\zeta_{L(\Zp)}^{*}(s)$ show the same behavior when $p$ varies. Furthermore, two class-2 Lie rings $L$ and $M$ are \textit{*-isospectral} over $\Zp$ if they are \textit{*-isospectral} over $\Fp$ 
\end{con}
\begin{con}
	If $L$ is a nilpotent class-c Lie ring, then $\zeta_{L(\Zp)}^{*}(s)$ and $\zeta_{L(\Zp/p^{c-1}\Zp)}^{*}(s)$ will show the same behavior when $p$ varies.  Furthermore, two class-c Lie rings $L$ and $M$ are \textit{*-isospectral} over $\Zp$ if they are \textit{*-isospectral} over $\Zp/p^{c-1}\Zp$ 
\end{con}
More examples to support or disprove these conjectures would be desirable.
\section*{Acknowledgments}
This article was partially supported by the National Research Foundation of Korea (NRF) grant funded by the Korean government (MEST), No. 2019R1A6A1A10073437. Some of the results in this article were parts of the author's Ph.D. thesis \cite{Lee/19thesis} from the University of Oxford. The author gratefully acknowledge inspiring mathematical discussions with Marcus du Sautoy, Benjamin Klopsch, Josh
Mag\-li\-one, Eamonn O'Brien, Dan Segal and Christopher Voll about
the research presented in this paper.

\bibliographystyle{amsplain} 
\bibliography{Lee_Fp_5OCT20_arXiv}

\end{document}